\documentclass{amsart}
\usepackage{graphicx}
\usepackage[all]{xy}
\usepackage{amscd}
\usepackage{amssymb}
\usepackage{amsmath}
\usepackage{amsthm}
\usepackage{amsfonts}
\usepackage{graphics}
\usepackage{epsfig,psfrag}
\usepackage[english]{babel}
\usepackage{latexsym}
\usepackage{mathrsfs}

\newtheorem{theorem}{Theorem}[section]

\newtheorem{lemma}[theorem]{Lemma}
\newtheorem*{claim}{Claim}
\newtheorem{proposition}[theorem]{Proposition}

\theoremstyle{definition}

\newtheorem{definition}[theorem]{Definition}

\theoremstyle{remark}
\newtheorem{remark}[theorem]{Remark}

\newcommand\co{\colon\thinspace}

\begin{document}

\title[Alternating diagrams and solenoids]
{Alternating Heegaard diagrams and Williams solenoid
attractors in 3--manifolds}

\author{Chao Wang}
\address{School of Mathematical Sciences, University of
Science and Technology of China, Hefei 230026, CHINA}
\email{chao\_{}wang\_{}1987@126.com}

\author{Yimu Zhang}
\address{Mathematics School, Jilin University,
Changchun 130012, CHINA}
\email{zym534685421@126.com}

\subjclass[2000]{Primary 57N10, 37C70, 37D45; Secondary 57M12}

\keywords{Heegaard diagram, solenoid attractor, Prism manifold,
Poincar\'e's homology 3--sphere, figure eight knot}

\begin{abstract}
We find all Heegaard diagrams with the property ``alternating'' or
``weakly alternating'' on a genus two orientable closed surface.
Using these diagrams we give infinitely many genus two 3--manifolds,
each admits an automorphism whose non-wondering set consists of two
Williams solenoids, one attractor and one repeller. These manifolds
contain half of Prism manifolds, Poincar\'e's homology 3--sphere and
many other Seifert manifolds, all integer Dehn surgeries on the
figure eight knot, also many connected sums. The result shows that
many kinds of 3--manifolds admit a kind of ``translation'' with
certain stability.
\end{abstract}

\date{}
\maketitle

\section{Introduction}
In \cite{Sm}, Smale introduced the solenoid attractor into dynamics
as an example of indecomposable hyperbolic non-wondering set. It has
a nice geometric model, namely the nested intersections of solid
tori. Suppose $f$ is a fibre preserving embedding from a disk fibre
bundle $N$ over $S^1$ into itself, contracting the fibres and
inducing an expansion on $S^1$, then $\bigcap_{i=1}^{\infty}f^i(N)$
is a so called Smale solenoid. To generalize this kind of
construction, in \cite{Wi}, Williams introduced solenoid attractors
derived from expansions on 1--dimensional branched manifolds. It
also has a geometric model, as the nested intersections of
handlebodies.

For a 3--manifold $M$, many of these attractors can be realized by
the geometric models with suitable automorphisms $f \in Diff(M)$.
But for most cases the realizations will not be global. Global means
that the non-wondering set $\Omega(f)$ is the union of solenoid
attractors and repellers. Here a repeller of $f$ is an attractor of
$f^{-1}$. By standard arguments in dynamics, one can show that if we
require $\Omega(f)$ consists of solenoid attractors and repellers,
then there must be exactly one attractor and one repeller. And $f$
is like a translation on $M$.

Motivated by the study in Morse theory and Smale's work in dynamics,
the following question was suggested in \cite{JNW} by Jiang, Ni and
Wang who studied this global realization question for Smale
solenoids.

{\bf Question :} When does a 3--manifold admit an automorphism whose
non-wandering set consists of solenoid attractors and repellers?

In \cite{JNW}, they showed that for a closed orientable 3--manifold
$M$, there is a diffeomorphism $f\co M\rightarrow M$ with the
non-wandering set $\Omega(f)$ a union of finitely many Smale
solenoids IF and ONLY IF $M$ is a Lens space $L(p, q)$ with $p \neq
0$, namely $M$ has Heegaard genus one and is not $S^1\times S^2$.
They also showed that the $f$ constructed in the IF part is
$\Omega$--stable, but is not structurally stable.

As in the opinion of \cite{JNW}, a manifold $M$ admitting a dynamics
$f$ such that $\Omega(f)$ consists of one hyperbolic attractor and
one hyperbolic repeller presents a symmetry of the manifold with
certain stability. The simplest example is the sphere, which admits
a dynamics $f$ such that $\Omega(f)$ consists of exactly two
hyperbolic fixed points, a sink and a source. Lens spaces give us
more such examples when we consider more complicated attractors. It
is believed by Jiang, Ni and Wang that many more 3--manifolds admit
such symmetries if we replace the Smale solenoids by the Williams
solenoids. As special cases, Wang asked whether the Poincar\'e's
homology 3--sphere admits such a symmetry? What about hyperbolic
3--manifolds?

Similar with the discussion in \cite{JNW}, in \cite{MY}, Ma and Yu
showed that for a closed orientable 3--manifold $M$, if there is a
$f \in Diff(M)$ such that $\Omega(f)$ consists of Williams
solenoids, whose defining handlebodies have genus $g\leq 2$, then
the Heegaard genus $g(M)\leq 2$. On the other hand, to construct
such $M$ and $f$, they introduced the alternating Heegaard splitting
which is a genus two splitting and admits a so called alternating
Heegaard diagram (see Definition \ref{AlterD}). They showed that if
$M$ admits an alternating Heegaard splitting, then there is a $f$
such that $\Omega(f)$ consists of two Williams solenoids, whose
defining handlebodies have genus two. As an interesting example,
they showed that the truncated-cube space (see \cite{M}), whose
fundamental group is the extended triangle group of order 48, admits
an alternating Heegaard splitting.

The motivation of this paper is to find further such examples. As
special cases, we will show that the Poincar\'e's homology 3--sphere
and many hyperbolic 3--manifolds admit such ``symmetries with
certain stability''. Hence we give a partial answer to the questions
asked by Wang.

Concretely, let $S^2(a,b,c)$ denote the Seifert fibred spaces, with
base $S^2$ and three singular fibres having invariants $a,b,c$. For
example, $S^2(-1/2,1/4,1/3)$ is the truncated-cube space. Let
$P(m,n)$ denote the manifolds $S^2(-1/2,1/2,m/n)$, which are the so
called Prism manifolds, the simplest 3--manifolds other than Lens
spaces.

\begin{theorem}\label{MAS}
For a 3--manifold $M$ in the following classes, it admits an
alternating Heegaard splitting.
\begin{itemize}
\item $P(m,n)$, $0<m<n, (m,n)=1$.
\item $S^2(-1/2,1/4,m/n)$, $0<m<n/2, (m,n)=1$.
\item $L(n,m)\,\#\, S^1\times S^2$, $L(n,m)\,\#\, RP^3$, $0\leq m<n,
(m,n)=1$.
\end{itemize}
Also there are infinitely many hyperbolic 3--manifolds admitting
such splittings. For these 3--manifolds there exist $f \in Diff(M)$
such that $\Omega(f)$ consist of two Williams solenoids.
\end{theorem}

In fact, we can find all the alternating Heegaard diagrams on a
genus two orientable surface. These diagrams can be determined by
integral vectors $(n,k_1,k_2,k_3)$, which satisfy $n>0$ and the
greatest common divisor $(n,k_1+k_2+2k_3)=1$. The 3--manifolds in
Theorem \ref{MAS} come from special diagrams.

On the other hand, having an alternating Heegaard splitting is a
strong restriction to genus two 3--manifolds. As it is pointed in
\cite{MY}, if $M$ admits an alternating Heegaard splitting, then
$H_1(M,\mathbb{Z}_2)\neq 0$. Hence we can not apply the result in
\cite{MY} to the Poincar\'e's homology 3--sphere. After a
modification, we generalize the alternating Heegaard splitting to
the weakly alternating Heegaard splitting (see Definition
\ref{WAlterD}), which also guarantees the existence of the required
$f$.

\begin{theorem}\label{wahws}
If the closed orientable 3--manifold $M$ admits a weakly alternating
Heegaard splitting, then there is a diffeomorphism $f \in Diff(M)$
such that $\Omega(f)$ consists of two Williams solenoids.
\end{theorem}

We can also find all the so called weakly alternating Heegaard
diagrams and for a part of them we can identify the corresponding
3--manifolds. Notice that the Poincar\'e's homology 3--sphere has
the form $S^2(-1/2,1/3,1/5)$. $\forall l\in \mathbb{Z}$, let
$S^3_{l/1}(4_1)$ denote the $l/1$--surgery on the figure eight knot.

\begin{theorem}\label{MASP}
For a 3--manifold $M$ in the following classes, it admits a weakly
alternating Heegaard splitting.
\begin{itemize}
\item $S^3_{l/1}(4_1)$.
\item $S^2(-1/2,1/l,m/n)$, $0<m<n, (m,n)=1$.
\item $S^2(1/l,1/r,1/n)$, $n>0$.
\item $L(n,m)\,\#\, L(l,1)$, $0\leq m<n, (m,n)=1$.
\end{itemize}
For these 3--manifolds there exist $f \in Diff(M)$ such that
$\Omega(f)$ consist of two Williams solenoids.
\end{theorem}

Here $l$ and $r$ can be all integers. In the second and third
classes if $l$ or $r$ is $0$, then we will get connected sums, not
Seifert fibred spaces. Notice that in each of the four classes there
are infinitely many 3--manifolds with $H_1(M,\mathbb{Z}_2)=0$.

By the same argument as in \cite{JNW}, one can show that all the $f$
we constructed are $\Omega$--stable, but are not structurally
stable. Theorem \ref{MAS} and \ref{MASP} convince us that there are
many more 3--manifolds admitting such ``symmetries with certain
stability''. Surely all the (weakly) alternating Heegaard diagrams
can give us many kinds of manifolds in the Thurston's picture of
3--manifolds. But at present we can only recognize a part of them.

In Section \ref{def}, we give some basic definitions, including the
handcuffs solenoid, alternating Heegaard diagram and alternating
Heegaard splitting. Then we give a brief introduction to the
construction of the $f \in Diff(M)$, appeared in \cite{JNW} and
\cite{MY}. Then we divide the proof of Theorem \ref{MAS} into two
steps:

In Section \ref{AD}, we will find all alternating Heegaard diagrams.

In Section \ref{DM}, we identify for special alternating Heegaard
diagrams which 3--manifolds they give, hence give a proof of Theorem
\ref{MAS}.

The discussion of weakly alternating Heegaard splitting (diagram)
will be parallel to the alternating case.

In Section \ref{Fexm}, we introduce weakly alternating Heegaard
splitting (diagram) and give a proof of Theorem \ref{wahws}. Then we
will find all weakly alternating Heegaard diagrams.

In Section \ref{WDM}, we identify for special weakly alternating
Heegaard diagrams which 3--manifolds they give, hence give a proof
of Theorem \ref{MASP}. In the end, we give some further remarks.

\section{Basic definitions and constructions}\label{def}

\subsection{Handcuffs solenoid and alternating Heegaard diagram}

All the Williams solenoids we considered will have the following
geometric model. For general definition and more details one can see
\cite{Wi}.

Let $N$ be a genus two handlebody with the $C^r(r\geq 1)$ ``disk
fibre bundle'' structure, fibred over the branched $C^r$ manifold
$K$, as in Figure \ref{fig:1}. Let $p$ denote the projection map
$N\rightarrow K$. We always suppose there is a Riemannian metric on
$N$.

\begin{figure}[h]
\centerline{\scalebox{0.5}{\includegraphics{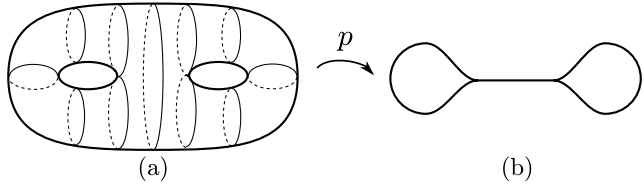}}} \caption{Disk
bundle and handcuffs}\label{fig:1}
\end{figure}

Suppose $f\co N\rightarrow N$ is a fibre preserving $C^r$ map such
that $f\co N\rightarrow f(N)$ is a diffeomorphism, and the induced
map $g\co K\rightarrow K$ is an immersion. We also require:

{\it Contracting condition on fibres:} for each fibre $D$, $f(D)$
lies in the interior of a fibre and
$\lim_{i\rightarrow\infty}Diameter(f^i(D))=0$.

{\it Expanding condition on $K$:} $g$ is an expansion and
$\Omega(g)=K$. More over, each point of $K$ has a neighborhood whose
image under $g$ is an arc.

Here the immersion $g$ is an expansion means that there is a
Riemannian metric ``$||\cdot||$'' on the tangent bundle $T(K)$ and
constants $C>0$, $\lambda>1$, such that $$||(Dg)^n(v)||\geq
C\lambda^n||v||, \forall n \in \mathbb{Z}^+, v\in T(K).$$

\begin{remark}
The Expanding condition can be required for self immersions of
general branched manifolds. In our case $K$ is like a handcuffs. Any
open set of $K$ will be mapped onto $K$ by $g^n$ for large $n$. Then
$g$ is an expansion implies $\Omega(g)=K$.
\end{remark}

Figure \ref{fig:2} is an example of such a $f$ and the corresponding
immersion $g$.

\begin{figure}[h]
\centerline{\scalebox{0.5}{\includegraphics{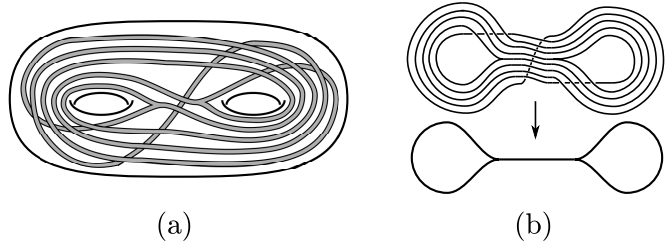}}}
\caption{Embedding and expansion}\label{fig:2}
\end{figure}

\begin{definition}
We call $\Lambda_f=\bigcap_{i=1}^{\infty}f^i(N)$ a handcuffs
solenoid with a defining neighborhood $N$ and a ``shift map''
$f|_{\Lambda_f}$.
\end{definition}

\begin{remark}
Let $\Sigma$ be the inverse limit of the sequence $K\leftarrow
K\leftarrow \cdots$ which is induced by $g$. For a point $a=(a_0,
a_1, a_2,\cdots)\in \Sigma$, we define $h(a)=(g(a_0), a_0,
a_1,\cdots)$. Then $h\co \Sigma \rightarrow \Sigma$ is a
homeomorphism. As the definition of Williams, $\Sigma$ is called the
solenoid with the shift map $h$. The dynamics $(\Lambda_f,
f|_{\Lambda_f})$ and $(\Sigma, h)$ are conjugate, by the
homeomorphism $P\co \Lambda_f\rightarrow \Sigma$, $P(x)=(p(x),
p(f^{-1}(x)), p(f^{-2}(x)), \cdots)$, $\forall x \in \Lambda_f$.
\end{remark}

\begin{definition}
A diagram $\mathcal{D}$ on an orientable closed surface $S$ is a
finite collection of simple closed curves intersecting transversely
in $S$.

Two diagrams $\mathcal{D}_1$ and $\mathcal{D}_2$ on $S$ are isotopic
if there is an isotopy of $S$ carries $\mathcal{D}_1$ to
$\mathcal{D}_2$. Isotopic diagrams will be thought as the same one.

Two diagrams $\mathcal{D}_1$ and $\mathcal{D}_2$ on $S$ are
homeomorphic, denoted by $\mathcal{D}_1\simeq\mathcal{D}_2$, if
there is a homeomorphism $h\co S\rightarrow S$ such that
$h(\mathcal{D}_1)=\mathcal{D}_2$.
\end{definition}

For any closed orientable 3--manifold $M$, there is an orientable
closed subsurface $S$ splitting $M$ into two handlebodies $N_1$ and
$N_2$. In this paper, we only consider the splitting with $S$ having
genus two. Hence for each $N_i$ we can find disjoint simple closed
curves $\alpha_i$, $\beta_i$, $\gamma_i$ in $S$ such that they all
bound disks in $N_i$, $\gamma_i$ is a separating curve, $\alpha_i$
and $\beta_i$ are non-separating and lie in different sides of
$\gamma_i$. Then $\{\alpha_1,\beta_1,\gamma_1\}$ together with
$\{\alpha_2,\beta_2,\gamma_2\}$ form a diagram on $S$.

\begin{definition}\label{AlterD}
We call the diagram
$\{\alpha_1,\beta_1,\gamma_1\}\cup\{\alpha_2,\beta_2,\gamma_2\}$ an
alternating Heegaard diagram if each curve of
$\{\alpha_i,\beta_i,\gamma_i\}$ intersects
$\{\alpha_j,\beta_j,\gamma_j\}$ in the cyclic order $$\alpha_j,
\gamma_j, \beta_j, \gamma_j, \alpha_j, \gamma_j, \beta_j, \gamma_j,
\cdots, i\neq j.$$

We call a Heegaard splitting alternating if it admits an alternating
Heegaard diagram.
\end{definition}

\begin{remark}
1. In the classical definition of Heegaard diagram, $\gamma_i$ may
be omitted.

2. The above definition of alternating Heegaard splitting coincides
with the definition of ``Alternating Heegaard splitting of type I''
in \cite{MY}.

3. If we just require $\{\alpha_1,\beta_1,\gamma_1\}$ are disjoint
simple closed curves in $S$ such that they intersect
$\{\alpha_2,\beta_2,\gamma_2\}$ as in Definition \ref{AlterD}, then
one can show that $\gamma_1$ must be separating, $\alpha_1$ and
$\beta_1$ are non-separating and lie in different sides of
$\gamma_1$.
\end{remark}

\begin{figure}[h]
\centerline{\scalebox{0.54}{\includegraphics{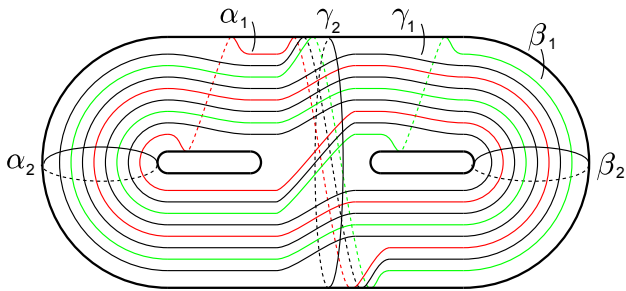}}}
\caption{Alternating Heegaard diagram}\label{fig:3}
\end{figure}

As an example, Figure \ref{fig:3} shows an alternating Heegaard
diagram. By the discussions in Section \ref{AD} and \ref{DM}, we
will see that this diagram gives us the Prism manifold $P(1,2)$.

\subsection{Construction of the diffeomorphism $f$}\label{consf}

Suppose $M=N_1\cup_{S} N_2$ is a genus two alternating Heegaard
splitting, having an alternating Heegaard diagram
$\{\alpha_1,\beta_1,\gamma_1\}\cup\{\alpha_2,\beta_2,\gamma_2\}$.
Then we can construct the required $f \in Diff(M)$ as following. For
more details one can see \cite{JNW} and \cite{MY}.

Firstly we give $N_i$ a ``disk fibre bundle'' structure, fibred over
the branched manifold $K$, such that $\alpha_i,\beta_i,\gamma_i$ are
all boundaries of fibres. Let $p_i$ be the corresponding projection
map. We choose a spine $K_i$ in $N_i$ as in Figure \ref{fig:w1}(a),
then $p_i|_{K_i}\co K_i\rightarrow K$ is an immersion as in Figure
\ref{fig:w1}(b).

\begin{figure}[h]
\centerline{\scalebox{0.6}{\includegraphics{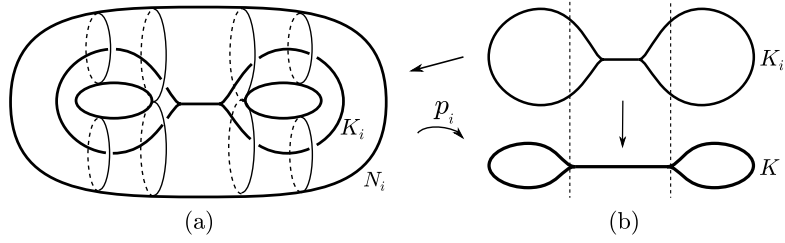}}}
\caption{Spine in $N_i$}\label{fig:w1}
\end{figure}

Then choose three points $x_\alpha,x_\beta,x_\gamma$ separately in
$\alpha_1\cap\alpha_2,\beta_1\cap\beta_2,\gamma_1\cap\gamma_2$, and
add three half twist bands between ``edges'' of $K_i$ and
$\alpha_j,\beta_j,\gamma_j$, $i\neq j$. The ``core'' of each band
should contain a chosen point and lie in the fibre. The half twists
from different sides should have the same ``direction''. Figure
\ref{fig:w2}(a) shows the three bands in $N_2$ and Figure
\ref{fig:w2}(b) shows that two bands from different sides intersect
at a chosen point.

\begin{figure}[h]
\centerline{\scalebox{0.6}{\includegraphics{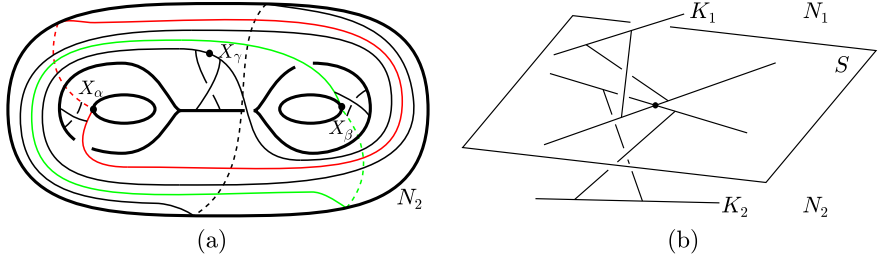}}}
\caption{Adding bands on spines}\label{fig:w2}
\end{figure}

We can get two new branched manifolds, and one of them is as in
Figure \ref{fig:w3}(a). We further push them into $N_i$ to get
$K_i'$ as in Figure \ref{fig:w3}(b). We can require that
$p_i|_{K_i'}\co K_i'\rightarrow K$ is also an immersion. Denote the
regular neighborhoods of $K_i$ and $K_i'$ by $N(K_i)$ and $N(K_i')$,
which are all contained in $N_i$ and have induced ``disk fibre
bundle'' structure. We construct the required $f \in Diff(M)$ in
three steps.

\begin{figure}[h]
\centerline{\scalebox{0.65}{\includegraphics{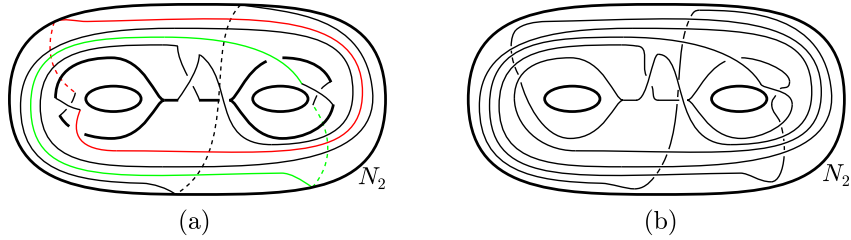}}} \caption{New
branched manifold in handlebody}\label{fig:w3}
\end{figure}

\noindent {\bf Step 1:} There is a $f_1 \in Diff(M)$ which is
isotopic to the identity, fixing $N(K_1')$ and on $N_2$ it satisfies
the {\it Contracting condition on fibres}, mapping $N_2$ to
$N(K_2)$, see Figure \ref{fig:w4}.

\begin{figure}[h]
\centerline{\scalebox{0.6}{\includegraphics{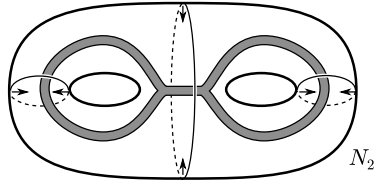}}}
\caption{Contraction on fibres}\label{fig:w4}
\end{figure}

\noindent {\bf Step 2:} Isotopy $K_2$ and its neighborhood $N(K_2)$
along the bands in $N_2$, see Figure \ref{fig:w5}(b). Since
$\alpha_2,\beta_2,\gamma_2$ bound disjoint disks in $N_2$, we can
then isotopy $K_1'$ and its neighborhood $N(K_1')$ along those
disks, see Figure \ref{fig:w5}(c). And we can further isotopy them
to the position as in Figure \ref{fig:w5}(d)

\begin{figure}[h]
\centerline{\scalebox{0.6}{\includegraphics{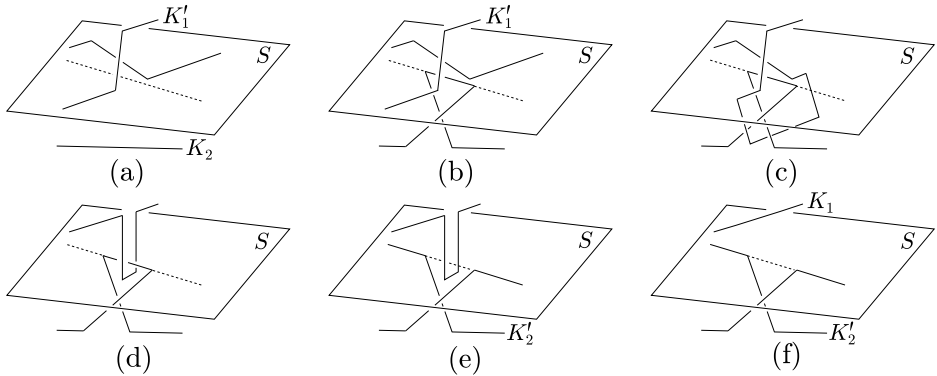}}}
\caption{Isotopy of $K_1$ and $K_2$}\label{fig:w5}
\end{figure}

Then since $\alpha_1,\beta_1,\gamma_1$ bound disjoint disks in
$N_1$, we can isotopy $K_2$ and $N(K_2)$ further along these disks
to $K_2'$ and $N(K_2')$, see Figure \ref{fig:w5}(e). And finally we
can isotopy $K_1'$ and $N(K_1')$ to $K_1$ and $N(K_1)$ to get $f_2
\in Diff(M)$, see Figure \ref{fig:w5}(f).

$f_2\mid_{N(K_1')}\co N(K_1')\rightarrow N(K_1)$ and
$f_2\mid_{N(K_2)}\co N(K_2)\rightarrow N(K_2')$ can be chosen to be
fibre preserving. If we let $g_1$ and $g_2$ denote their induced
maps on $K$, then $f_2$ can be further chosen such that $g_1^{-1}$
and $g_2$ satisfy the {\it Expanding condition on $K$}.

\begin{figure}[h]
\centerline{\scalebox{0.6}{\includegraphics{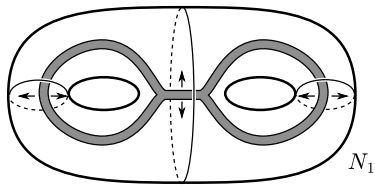}}}
\caption{Expansion on fibres}\label{fig:w6}
\end{figure}

\noindent {\bf Step 3:} There is a $f_3 \in Diff(M)$ which is
isotopic to the identity, fixing $N(K_2')$ and on $N_1$ its inverse
$f_3^{-1}$ satisfies the {\it Contracting condition on fibres},
mapping $N_1$ to $N(K_1)$. On $N(K_1)$ the map $f_3$ is as in Figure
\ref{fig:w6}.

Let $f=f_3\circ f_2\circ f_1 \in Diff(M)$, by the construction $f$
is isotopic to the identity. It is easy to see
$\Omega(f)=\bigcap_{i=1}^{\infty}f^i(N_2)\bigcup\bigcap_{i=1}^{\infty}f^{-i}(N_1)$
is the union of two Williams solenoids. And clearly the Williams
solenoids derived from alternating Heegaard splittings (defined as
in Definition \ref{AlterD}) are all handcuffs solenoids.

\section{Alternating Heegaard diagram}\label{AD}
Suppose
$\{\alpha_1,\beta_1,\gamma_1\}\cup\{\alpha_2,\beta_2,\gamma_2\}$ is
an alternating Heegaard diagram on a splitting surface $S$. We can
assume the curves $\{\alpha_2,\beta_2,\gamma_2\}$ are in the
standard position like in Figure \ref{fig:3}. We color the curves
$\{\alpha_1,\beta_1,\gamma_1\}$ separately by Red, Green and Black.
Then the Red(Green) curve is non-separating, the Black curve is
separating.

\begin{figure}[h]
\centerline{\scalebox{0.92}{\includegraphics{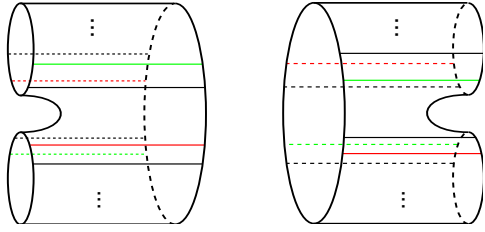}}} \caption{Two
3--punctured spheres}\label{fig:4}
\end{figure}

Cutting $S$ along $\{\alpha_2,\beta_2,\gamma_2\}$, we get two
3--punctured spheres $S_l$ and $S_r$. Since
$\{\alpha_1,\beta_1,\gamma_1\}$ intersect
$\{\alpha_2,\beta_2,\gamma_2\}$ in the cyclic order $\alpha_2$,
$\gamma_2$, $\beta_2$, $\gamma_2,\cdots$, the colored curves must be
cut into arcs lying in $S_l$ and $S_r$. And it can be
``straightened'' as in Figure \ref{fig:4}. Clearly colored arcs in
$S_l$ and $S_r$ have the same number. Since
$\{\alpha_2,\beta_2,\gamma_2\}$ intersect
$\{\alpha_1,\beta_1,\gamma_1\}$ in the cyclic order $\alpha_1$,
$\gamma_1$, $\beta_1$, $\gamma_1,\cdots$, this number can be divided
by $8$.

The original diagram can be obtained from Figure \ref{fig:4} by
pasting the cuts. There is a quite natural way to paste the cuts as
in Figure \ref{fig:5} which contains $4n(n>0)$ (non-colored)
parallel simple closed curves. Hence the original diagram can be
thought as obtained from Figure \ref{fig:5} by some ``twist''
operations.

\begin{figure}[h]
\centerline{\scalebox{0.48}{\includegraphics{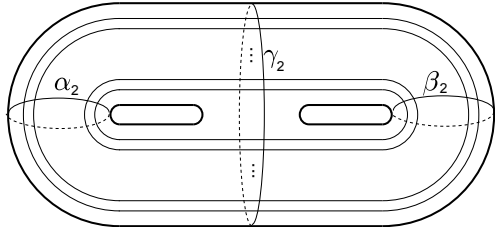}}}
\caption{Trivial diagram}\label{fig:5}
\end{figure}

\begin{definition}
Let $\mathcal{D}$ be a diagram on an oriented closed surface $S$,
$c$ is a simple closed curve in $S$ which intersects $\mathcal{D}$
transversely. Then we have a local picture as in Figure
\ref{fig:6}(a). The twist operation $\mathcal{T}_c$ on $\mathcal{D}$
is as in Figure \ref{fig:6}(b). It is invertible.

\begin{figure}[h]
\centerline{\scalebox{0.6}{\includegraphics{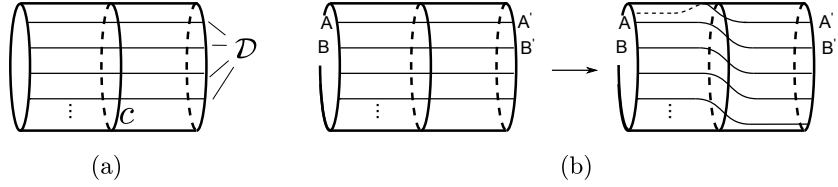}}} \caption{Local
model and twist operation}\label{fig:6}
\end{figure}

If we look from positive side of $S$, $\mathcal{T}_c$ is a right
hand shift along $c$, and $\mathcal{T}_c^{-1}$ is a left hand shift
along $c$.
\end{definition}

\begin{remark}\label{dehn}
$\mathcal{T}_c$ is an operation on diagrams. Do not confuse it with
the Dehn twist $t_c$, which is an automorphism of $S$ and normally
can be defined as in Figure \ref{fig:DT}. Out of the annulus
neighborhood of $c$, $t_c$ is the identity. On the annulus $t_c$ is
like a left hand $2\pi$--twist. Its inverse $t_c^{-1}$ is like a
right hand $2\pi$--twist.

\begin{figure}[h]
\centerline{\scalebox{0.6}{\includegraphics{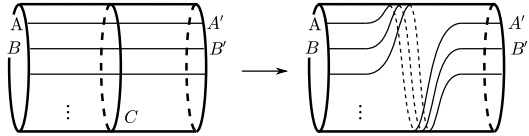}}}
\caption{Dehn twist}\label{fig:DT}
\end{figure}
\end{remark}

\begin{definition}
Define $D(4n;0,0,0)$ to be the diagram as in Figure \ref{fig:5}
which consists of $4n$ parallel curves and
$\{\alpha_2,\beta_2,\gamma_2\}$. Pushing each curve of
$\{\alpha_2,\beta_2,\gamma_2\}$ sightly to either side we get their
parallel curves $\{c_1,c_2,c_3\}$. Then $\mathcal{T}_{c_i}$ are
mutually commutative. Define $D(4n;m_1,m_2,m_3)$ to be
$\mathcal{T}_{c_1}^{m_1}\mathcal{T}_{c_2}^{m_2}\mathcal{T}_{c_3}^{m_3}(D(4n;0,0,0))$,
here $m_i$ are all integers.
\end{definition}

From the above discussion we have:

\begin{lemma}
Any alternating Heegaard diagram has the form $D(4n;m_1,m_2,m_3)$.
\end{lemma}

\begin{figure}[h]
\centerline{\scalebox{0.4}{\includegraphics{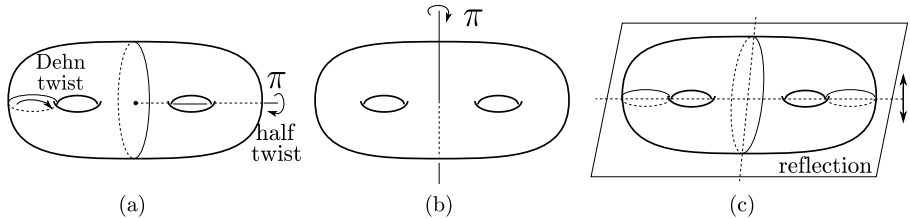}}}
\caption{Symmetries of diagrams}\label{fig:7}
\end{figure}

\begin{lemma}\label{sym}
For $D(4n;m_1,m_2,m_3)$, we have following homeomorphisms:

1. $D(4n;m_1,m_2,m_3)\simeq D(4n;m_1',m_2',m_3')$, $m_i\equiv m_i'
(mod\,4n)$.

2. $D(4n;m_1,m_2,m_3)\simeq D(4n;m_2,m_1,m_3)$.

3. $D(4n;m_1,m_2,m_3)\simeq D(4n;-m_1,-m_2,-m_3)$.
\end{lemma}

\begin{proof}
We can put arcs in $S_l$ and $S_r$ in a symmetric way as in Figure
\ref{fig:4}, then paste the cuts ``symmetrically'' to obtain the
diagrams. These homeomorphisms can be obtained by Dehn twist(half
twist), $\pi$--rotation and reflection as in Figure \ref{fig:7}.
\end{proof}

We also have the following lemma which can be easily proved.

\begin{lemma}\label{mirror}
The Dehn twist(half twist), $\pi$--rotation and reflection as in
Figure \ref{fig:7} map an alternating Heegaard diagram to an
alternating Heegaard diagram.
\end{lemma}


\begin{theorem}\label{all ad}
The diagram $D(4n;m_1,m_2,m_3)$ is an alternating Heegaard diagram
if and only if $(m_1,m_2,m_3)=\eta+4(k_1,k_2,k_3)$, here $\eta$ is
one of the following integral vectors $\pm(1,-3,1)$, $\pm(1,-5,2)$,
$k_i$ are all integers and satisfy $(n,k_1+k_2+2k_3)=1$.
\end{theorem}

\begin{proof}
{\bf The Only If Part:}

Suppose $D(4n;m_1,m_2,m_3)$ is alternating. Cut $S$ along
$\{\alpha_2,\beta_2,\gamma_2\}$ as before, then we get $S_l$ and
$S_r$. We first look at $S_l$.

\begin{figure}[h]
\centerline{\scalebox{0.63}{\includegraphics{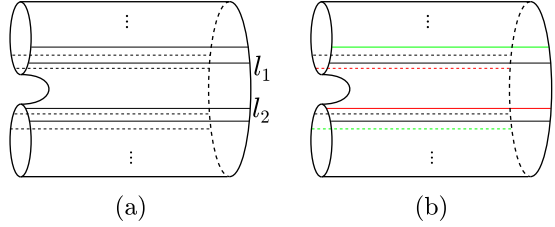}}}
\caption{Uncolored and colored left surfaces}\label{fig:8}
\end{figure}

By Definition \ref{AlterD}, it is easy to see that one of $l_1$ and
$l_2$ must be Black. By Lemma \ref{sym} and \ref{mirror}, we can
assume $l_1$ is Black, otherwise we consider the reflection image of
this diagram. Then we can further assume $l_2$ is Red, otherwise we
recolor the curves $\alpha_1$ and $\beta_1$. Hence by Definition
\ref{AlterD}, $S_l$ should be as in Figure \ref{fig:8}(b).

\begin{figure}[h]
\centerline{\scalebox{0.63}{\includegraphics{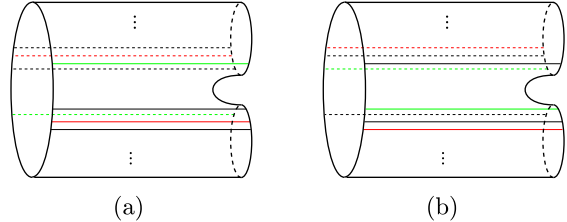}}} \caption{Two
possible right surfaces}\label{fig:9}
\end{figure}

The situation of $S_r$ should be similar. If we cut $S_l$ and $S_r$
further along all Black arcs, then only the piece containing the
saddle can contain two arcs with the same color Red or Green. Other
pieces are all rectangles containing only one arc. Hence for $S_r$
the piece containing the saddle must contain two Green arcs,
otherwise $\beta_1$ will be parallel to $\gamma_1$. Then by
Definition \ref{AlterD} we have two possibilities of $S_r$ as in
Figure \ref{fig:9}.

{\bf Case 1:} $S_r$ is as in Figure \ref{fig:9}(a).

We fix a base position $\eta=(1,-3,1)$. This means that if before
cutting along $\{\alpha_2,\beta_2,\gamma_2\}$ the diagram is
$D(4n;1,-3,1)$, then colors of the arcs will be coincident at the
cuts. $D(4n;m_1,m_2,m_3)$ can be obtained from $D(4n;1,-3,1)$ by
twist operations, hence clearly $(m_1,m_2,m_3)=\eta+4(k_1,k_2,k_3)$.

In $S$, colored curves intersect $\gamma_2$ at $8n$ points, $2n$ Red
points, $2n$ Green points and $4n$ Black points, along $\gamma_2$ in
the cyclic order Red, Black, Green, Black, $\cdots$. Looking at
$\gamma_2$ from left to right, give the Red(Green) points which
belong to the saddle piece a symbol $0$($0'$) and other Red(Green)
points symbols $1,2,\cdots,n-1$($1',2',\cdots,n-1'$) clockwise, then
the picture will be as in Figure \ref{fig:10}(a), $\bar{k}_3\equiv
k_3(mod\, n)$, $0\leq\bar{k}_3<n$.

\begin{figure}[h]
\centerline{\scalebox{0.9}{\includegraphics{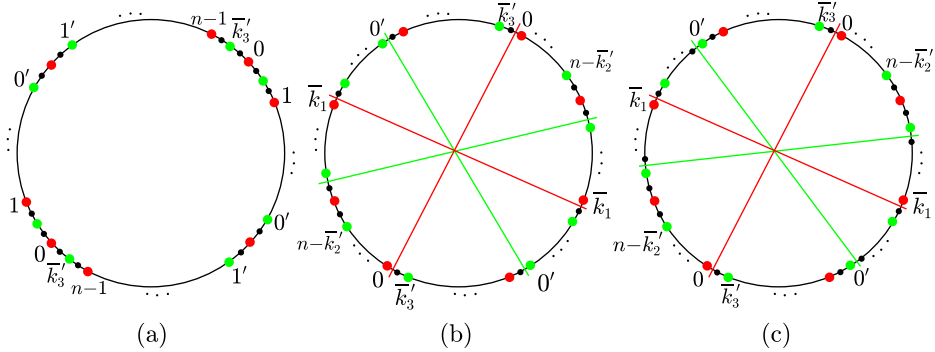}}}
\caption{Equivalence relation on points in $\gamma_2$}\label{fig:10}
\end{figure}

We define an equivalence relation on the Black(Red, Green) points in
$\gamma_2$, which is generated by the following two relations:

$\widetilde{R}_l$: two Black(Red, Green) points are equivalent if
arcs in $S_l$ containing them have a common boundary in $\alpha_2$.

$\widetilde{R}_r$: two Black(Red, Green) points are equivalent if
arcs in $S_r$ containing them have a common boundary in $\beta_2$.

\begin{figure}[h]
\centerline{\scalebox{0.8}{\includegraphics{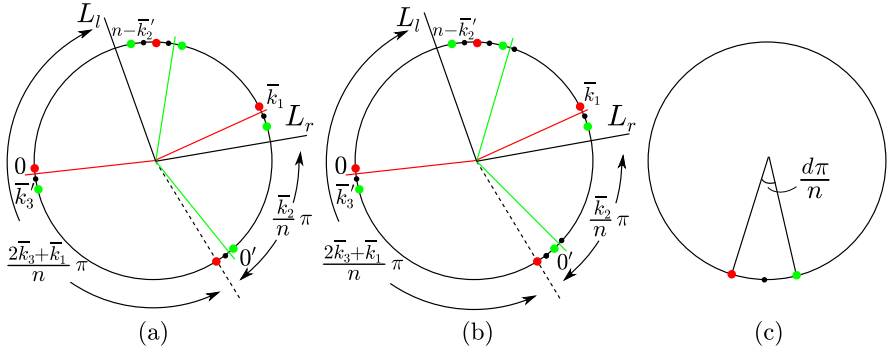}}}
\caption{Reflections and quotient space}\label{fig:11}
\end{figure}

There are four open sectors with Red(Green) boundaries as in Figure
\ref{fig:10}(b). Here $\bar{k}_i\equiv k_i(mod\, n)$,
$0<\bar{k}_2\leq n$, $0\leq\bar{k}_1,\bar{k}_3<n$, the Red(Green)
lines pass through the midpoints of the Red(Green) points and their
neighbor Black points. It can be checked that reflections on
$\gamma_2$ which interchange two Red(Green) non-adjacent sectors
give us $\widetilde{R}_l(\widetilde{R}_r)$ on the Black(Red, Green)
points in those sectors.

Then the equivalence relation induces an equivalence relation on
Black(Red, Green) points in $RP^1$. Here the $RP^1$ is obtained by
identifying antipodal points of $\gamma_2$. This induced equivalence
relation is generated by two reflections $R_l$ and $R_r$ with
reflection lines $L_l$ and $L_r$ as in Figure \ref{fig:11}(a).

By the connectedness of the Black(Red, Green) curve, all Black(Red,
Green) points in $\gamma_2$ are equivalent. Hence the dihedral group
generated by $R_l$ and $R_r$ acts transitively on the Black(Red,
Green) points in $RP^1$. In Figure \ref{fig:11}(a), if we let
$\theta$ denote the angle between $L_l$ and $L_r$, then we have
$\theta\equiv\pm(k_1+k_2+2k_3)\pi/n (mod\,\pi)$. Notice that $L_l$
and $L_r$ only pass through Red or Green points.

\begin{claim}
The group generated by $R_l$ and $R_r$ acts transitively on
Black(both Red and Green) points in $RP^1$ if and only if $(n,
k_1+k_2+2k_3)=1$.
\end{claim}

\begin{proof}[Proof of Claim] Let $(n, k_1+k_2+2k_3)=d$, then after
modular the group action we get a corner with boundaries contain Red
or Green points and having angle $d\pi/n$. Hence the group acts
transitively on Black(both Red and Green) points if and only if
$d=1$, see Figure \ref{fig:11}(c).
\end{proof}

Hence we finish the discussion of {\bf Case 1}.

{\bf Case 2:} $S_r$ is as in Figure \ref{fig:9}(b).

We fix a base position $\eta=(1,-5,2)$ similar to {\bf Case 1}. Then
as above we have $(m_1,m_2,m_3)=\eta+4(k_1,k_2,k_3)$. The following
discussion is exactly the same as in {\bf Case 1}, except that
instead of Figure \ref{fig:10}(b) and Figure \ref{fig:11}(a) we will
get Figure \ref{fig:10}(c) and Figure \ref{fig:11}(b), and we have
$(n, k_1+k_2+2k_3)=1$.

{\bf The If Part:}

Suppose $(m_1,m_2,m_3)=\eta+4(k_1,k_2,k_3)$, here $\eta$ is one of
$\pm(1,-3,1)$, $\pm(1,-5,2)$ and $(n,k_1+k_2+2k_3)=1$. By
Lemma\ref{sym} and \ref{mirror}, we can assume $\eta$ is $(1,-3,1)$
or $(1,-5,2)$. Cut $D(4n;m_1,m_2,m_3)$ along
$\{\alpha_2,\beta_2,\gamma_2\}$, then we get $S_l$ and $S_r$.

Clearly we can color the arcs in $S_l$ as in Figure \ref{fig:8}(b).
Then we can color $S_r$ as in Figure \ref{fig:9}(a) or Figure
\ref{fig:9}(b) according to $\eta$ is $(1,-3,1)$ or $(1,-5,2)$. Then
the colors of those arcs will coincide at points in
$\{\alpha_2,\beta_2,\gamma_2\}$. And we can have equivalence
relations on Black(Red, Green) points in $\gamma_2$ and Black(Red,
Green) points in $RP^1$ as in the proof of {\bf The Only If Part}.

Since $(n,k_1+k_2+2k_3)=1$, by the {\bf Claim} all Black(Red, Green)
points in $RP^1$ are equivalent. Hence there are at most two
equivalence classes of Black(Red, Green) points in $\gamma_2$. And
in $S$ we have at most two Black(Red, Green) curves.

\begin{figure}[h]
\centerline{\scalebox{0.45}{\includegraphics{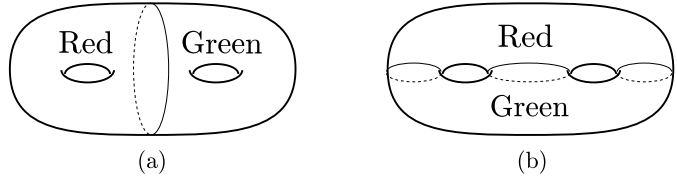}}} \caption{Red
and Green surfaces}\label{fig:12}
\end{figure}

Notice that there is a pair of Red(Green) antipodal points in
$\gamma_2$ lying in the saddle piece of $S_l(S_r)$. Hence the union
of pieces containing the Red(Green) arcs is a connected subsurface
in $S$, with Euler characteristic $-1$. Hence there are two possible
cases as in Figure \ref{fig:12}.

Since there are at most two Black curves, we meet the case Figure
\ref{fig:12}(a), and there is only one Black curve which is
separating. Then the Red(Green) curve is non-separating because the
two sides of it can be connected by a parallel curve of the Black
curve. Hence there are only one Red curve and one Green curve, both
non-separating.
\end{proof}

\section{Manifolds with alternating Heegaard
splittings}\label{DM}

\begin{definition}\label{admani}
Let $\eta_1=(1,-3,1)$, $\eta_2=(1,-5,2)$. Define
$M_i(n;k_1,k_2,k_3)$ to be the 3--manifold which has an alternating
Heegaard diagram $D(4n;m_1,m_2,m_3)$ with
$(m_1,m_2,m_3)=\eta_i+4(k_1,k_2,k_3)$, $i=1,2$. Here $n>0$ and
$(n,k_1+k_2+2k_3)=1$.
\end{definition}

\begin{lemma}\label{pare}
If a 3--manifold $M$ admits an alternating Heegaard splitting, then
$M$ must be homeomorphic to some $M_i(n;k_1,k_2,k_3)$ with the
inequalities $0<k_2\leq n$, $0\leq k_3<n$ and $n\leq
k_1+k_2+2k_3<2n$.
\end{lemma}

\begin{proof}
If two Heegaard diagrams are homeomorphic, then they give the
homeomorphic 3--manifolds. Then by Theorem \ref{all ad} and Lemma
\ref{sym} we get the results.
\end{proof}

Following we identify some of $M_1(n;k_1,k_2,k_3)$ and
$M_2(n;k_1,k_2,k_3)$ as in Lemma \ref{pare} to our familiar
3--manifolds. Notice that every alternating Heeagaard diagram admits
an involution $\tau$ which preserves the Black(Red, Green) curve, as
in Figure \ref{fig:13}.

\begin{figure}[h]
\centerline{\scalebox{0.4}{\includegraphics{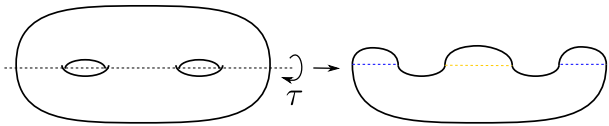}}}
\caption{Involution and branched cover}\label{fig:13}
\end{figure}

\begin{proposition}\label{branch}
The $M_1(n;k_1,k_2,k_3) (M_2(n;k_1,k_2,k_3))$ as in Lemma \ref{pare}
is a 2--fold branched cover of $S^3$. The branched set is a three
bridge link. It consists of a Blue two bridge link and a Yellow
trivial circle as in Figure \ref{fig:14}(a)(Figure \ref{fig:14}(b)).

\begin{figure}[h]
\centerline{\scalebox{0.55}{\includegraphics{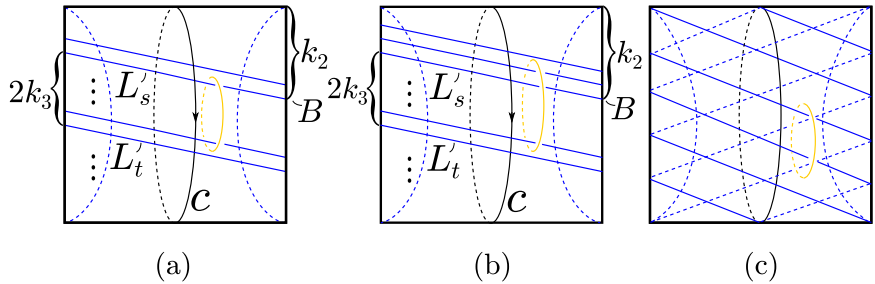}}}
\caption{Three bridge links}\label{fig:14}
\end{figure}

The front(back) Blue arcs lying in the surface of the $n\times n$
square pillow have slope $-m/n(m/n)$, $m=k_1+k_2+2k_3-n$. In the
front square, walking from the point $B$ to left we get the arc
$L_s$. Walking along the oriented circle $c$ from $L_s$ by $2k_3$ we
get the arc $L_t$. And then the position of the Yellow circle can be
determined. As an example, Figure \ref{fig:14}(c) shows the
corresponding branched set of $M_1(5;2,3,1)$.
\end{proposition}

\begin{proof}
Let $M_1(n;k_1,k_2,k_3)=N_1\cup_{S}N_2$ as before. On $N_2$ the
branched cover is given by the involution $\tau$ as in Figure
\ref{fig:13}. It induces a branched cover of the Black(Red, Green)
curve. On $\gamma_2$ it is a $\pi$--rotation and on $\alpha_2$ and
$\beta_2$ it is a reflection. These reflections are essentially the
$R_l$ and $R_r$ defined on $RP^1$ in the proof of Theorem \ref{all
ad}, see Figure \ref{fig:11}. Since the reflection lines $L_l$ and
$L_r$ only pass Red or Green points, we know that on the Black curve
$\tau$ is a $\pi$--rotation and on the Red(Green) curve $\tau$ is a
reflection.

Cut $N_2$ open along disks bounded by $\alpha_2$ and $\beta_2$.
Modular the involution, then we get a cylinder as in Figure
\ref{fig:15}(a). Then we can paste the left and right disks by
modular the reflections to get $N_2/\tau$, an $n\times n$ square
pillow as in Figure \ref{fig:15}(b). With suitable twists, we can
require that the front arcs have slope $-m/n$ and the back arcs have
slope $m/n$. And the position of the Yellow arc is as in Figure
\ref{fig:14}(a).

\begin{figure}[h]
\centerline{\scalebox{1.7}{\includegraphics{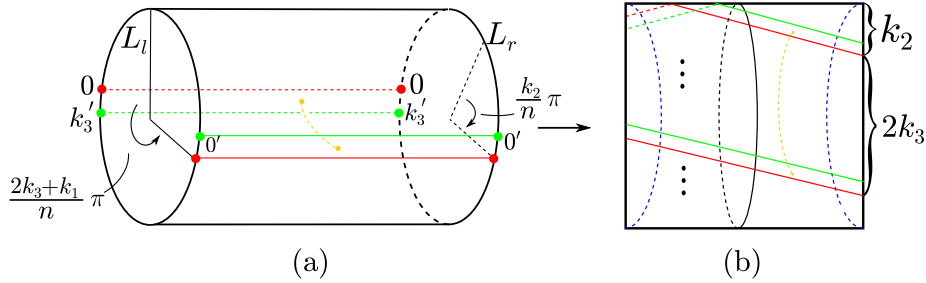}}}
\caption{Cylinder and square pillow}\label{fig:15}
\end{figure}

Let $N_1$ be as in Figure \ref{fig:16}(a). We can extend $\tau$ to a
$\pi$--rotation(reflection) on the disk bounded in $N_1$ by the
Black(Red, Green) curve. Hence we can further extend $\tau$ to the
whole $N_1$, and get the $N_1/\tau$ as in Figure \ref{fig:16}(b).

\begin{figure}[h]
\centerline{\scalebox{0.44}{\includegraphics{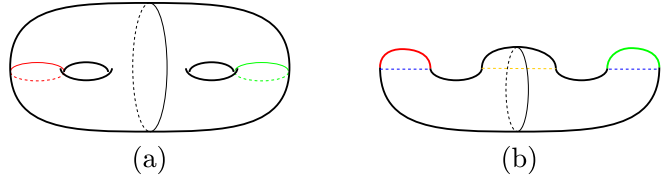}}}
\caption{Involution on $N_1$}\label{fig:16}
\end{figure}

Clearly $M_1(n;k_1,k_2,k_3)/\tau=N_1/\tau\cup_{S/\tau}N_2/\tau$ is a
$S^3$ with branched set a three bridge link that consists of a Blue
link and a Yellow circle. We can push the Blue arcs in $N_2/\tau$
across the disks to the Red and Green arcs, then the Yellow arc is
just a trivial arc in $N_2/\tau$. For $M_2(n;k_1,k_2,k_3)$ the
discussion is similar, and we finish the proof.
\end{proof}

\begin{proposition}\label{Idmfd}
Suppose $M_1(n;k_1,k_2,k_3)(M_2(n;k_1,k_2,k_3))$ is as in Lemma
\ref{pare} and $m=k_1+k_2+2k_3-n$. We have the following
homeomorphisms:

1. $M_1(n;k_1,k_2,0)\simeq L(n,m)\,\#\,S^1\times S^2$, $0\leq m<n$.

2. $M_2(n;k_1,k_2,0)\simeq L(n,m)\,\#\,L(2,1)$, $0\leq m<n$.

3. $M_1(n;0,n-m,m)\simeq P(m,n)$, $0<m<n$.

4. $M_2(n;m-1,n-2m+1,m)\simeq S^2(-1/2,1/4,m/n)$, $0<m<n/2$.
\end{proposition}

\begin{proof}
The proof depends on Proposition \ref{branch} and the fact that the
2--fold branched cover of a Montesinos link is a Seifert fibred
space. Moreover, a $(m,n)$--rational tangle corresponds to a
singular fibre with invariant $m/n$. This can be found, for example,
in Chapter 11 and 12 of \cite{BZ}.

Following we identify the 2--fold branched cover of the
corresponding links of $M_1$ and $M_2$ in the Proposition.
Considering Figure \ref{fig:14}, since the Yellow arc in $N_1/\tau$
is a trivial arc, we can push it into $S/\tau$ disjoint from the
Blue arcs. Then we further push it into the square pillow. Hence it
is contained in a smaller box, as in Figure \ref{fig:17}(a).

\begin{figure}[h]
\centerline{\scalebox{0.45}{\includegraphics{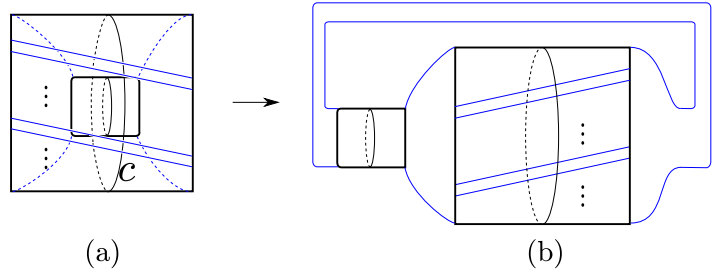}}} \caption{Two
boxes}\label{fig:17}
\end{figure}

After a $\pi$--rotation about the circle $c$ we change the outside
and inside of the square pillow, and we redraw it as in Figure
\ref{fig:17}(b). Now the left box contains the Yellow circle and two
Blue arcs, and the right box is exactly a $(m,n)$--rational tangle.

\begin{figure}[h]
\centerline{\scalebox{0.4}{\includegraphics{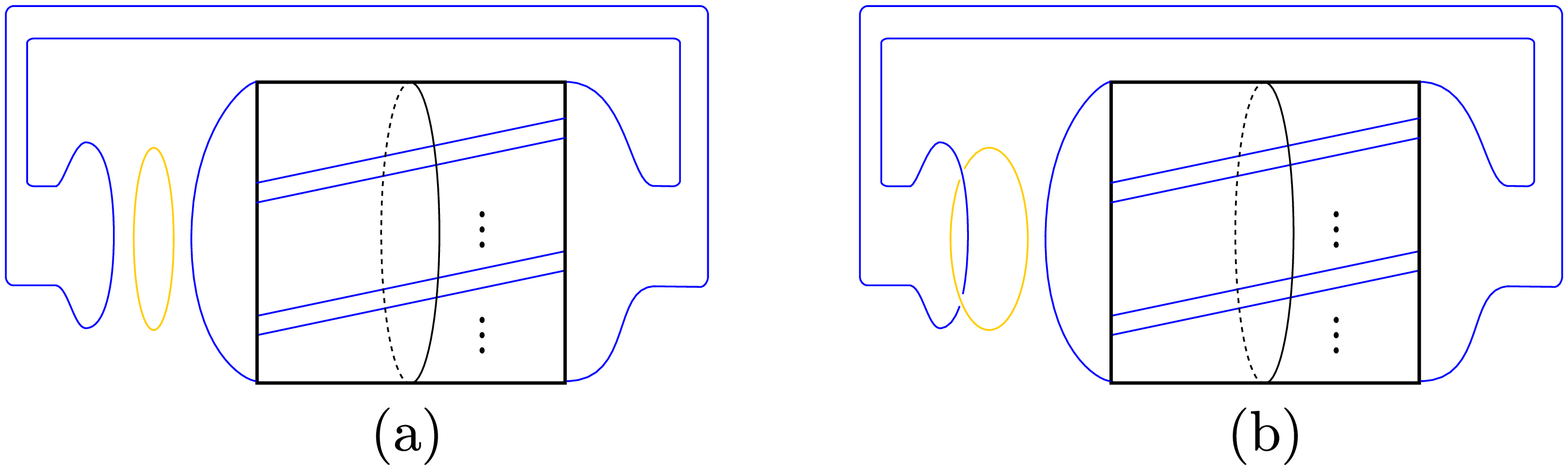}}}
\caption{Connected sums}\label{fig:18}
\end{figure}

When $k_3=0$, the picture is as in Figure \ref{fig:18}. The three
bridge link can be written as a connected sum of a two bridge link
and a 2--component trivial link (or a Hopf link). The connected sum
of links corresponds to the connected sum of their 2--fold branched
covers. The 2--fold branched cover of a 2--component trivial link
(or a Hopf link) is $S^1\times S^2$ (or $RP^3$). And the 2--fold
branched cover of the blue two bridge link in Figure \ref{fig:18}(a)
is $L(n,m)$. Hence we get the first two homeomorphisms.

To show the last two homeomorphisms, we redraw the corresponding
links in Figure \ref{fig:20}. Figure \ref{fig:20}(a)(c) give us the
pictures when we push the Yellow arc in $N_1/\tau$ into $S/\tau$.
Figure \ref{fig:20}(b)(d) show us how the links will be look like
after we do the procedure as in Figure \ref{fig:17}.

\begin{figure}[h]
\centerline{\scalebox{0.5}{\includegraphics{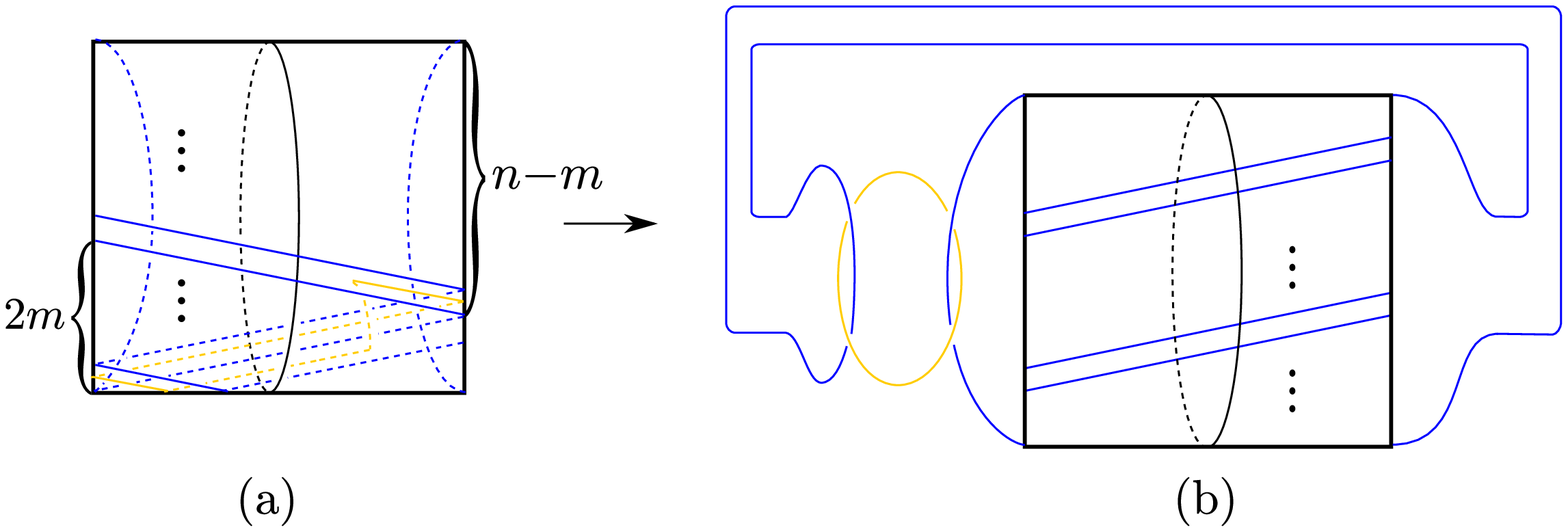}}}
\end{figure}

\begin{figure}[h]
\centerline{\scalebox{0.5}{\includegraphics{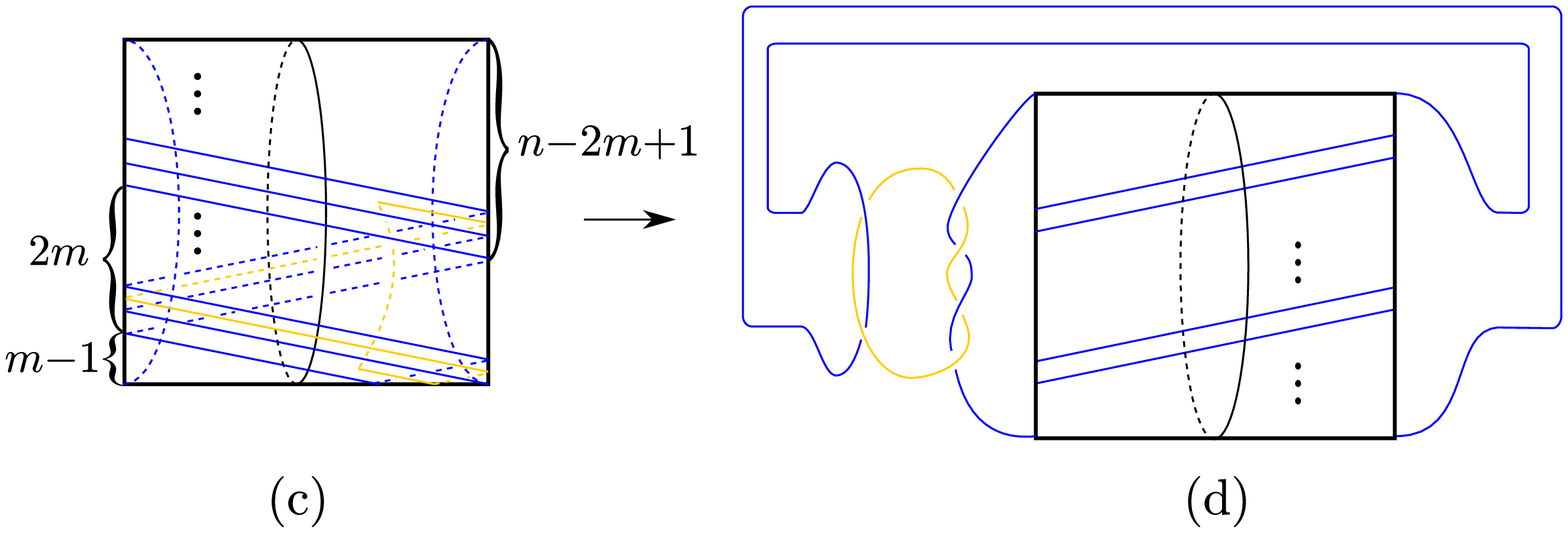}}}
\caption{Montesinos links}\label{fig:20}
\end{figure}

These two links are all Montesinos links with three rational
tangles. For Figure \ref{fig:20}(b), the three rational tangles have
parameters $(-1,2)$, $(1,2)$ and $(m,n)$. And for Figure
\ref{fig:20}(d), the three rational tangles have parameters
$(-1,2)$, $(1,4)$ and $(m,n)$. Hence 2--fold branched covers of
these two links are all Seifert fibred spaces, and the invariants
are exactly as in the Proposition.
\end{proof}

\begin{proposition}\label{hyperbolicM}
$M_2(n;0,n-3,2)(n\geq 5)$ has a 2--fold cover which is homeomorphic
to some Dehn surgery on the hyperbolic link $6^2_3$.
$M_2(n;0,n-3,2)$ are all hyperbolic 3--manifolds, except for
finitely many $n$.
\end{proposition}

\begin{figure}[h]
\centerline{\scalebox{0.5}{\includegraphics{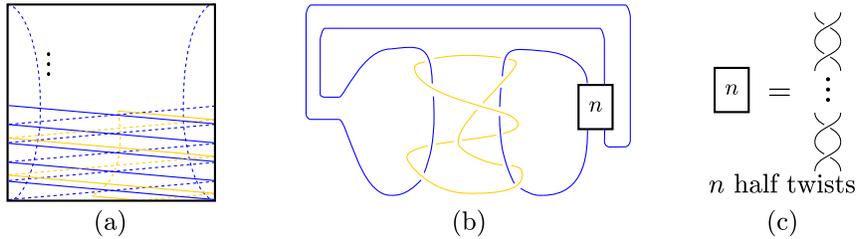}}} \caption{The
quotient $M_2(n;0,n-3,2)/\tau$}\label{fig:h1}
\end{figure}

\begin{proof}
The manifold $M_2(n;0,n-3,2)(n\geq 5)$ is the 2--fold branched cover
of the link as in Figure \ref{fig:h1}(a). As in the proof of
Proposition \ref{Idmfd}, we can isotopy it to Figure
\ref{fig:h1}(b), here the $n$--box denotes two parallel vertical
singular arcs with $n$ half twists as in Figure \ref{fig:h1}(c).

If we replace the $n$--box by a box containing two parallel
horizontal singular arcs, then the picture will be as in Figure
\ref{fig:h2}(a), which is a Hopf link. The new box can be thought as
a regular neighborhood of a regular arc. We can isotopy this picture
to the position as in Figure \ref{fig:h2}(b).

\begin{figure}[h]
\centerline{\scalebox{0.5}{\includegraphics{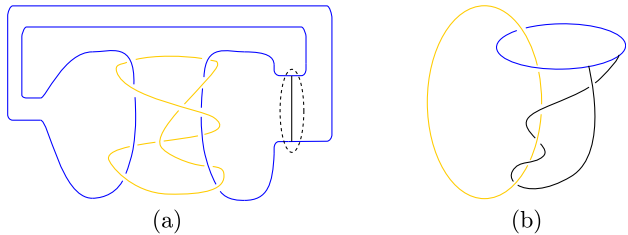}}}
\caption{Surgery on $M_2(n;0,n-3,2)/\tau$}\label{fig:h2}
\end{figure}

Clearly the 2--fold branched covers of the new box and the original
$n$--box are solid tori. Since the 2--fold branched cover of the
Hopf link is $RP^3$, we know that $M_2(n;0,n-3,2)$ is some Dehn
surgery on a knot in $RP^3$. When we consider a further 2--fold
cover, the knot become the link $6^2_3$ in $S^3$.

\begin{figure}[h]
\centerline{\scalebox{0.43}{\includegraphics{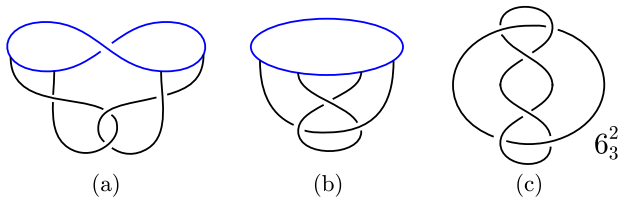}}}
\caption{Branched covers and link $6^2_3$}\label{fig:h3}
\end{figure}

This can be easily seen from another way to get the 4--fold branched
cover as following. Figure \ref{fig:h3}(a) is the 2--fold branched
cover of Figure \ref{fig:h2}(b). Figure \ref{fig:h3}(b) is isotopic
to Figure \ref{fig:h3}(a). And Figure \ref{fig:h3}(c) is the 2--fold
branched cover of Figure \ref{fig:h3}(b).

The link $6^2_3$ is hyperbolic, and one can show that its quotient
knot in $RP^3$ is also hyperbolic. Then by the Thurston's Hyperbolic
Dehn Surgery Theorem, all the surgeries are hyperbolic 3--manifolds,
except for finitely many cases (see \cite{Th}).
\end{proof}

\begin{remark}
Now the orbifold $M_2(n;0,n-3,2)/\tau$ has 1--dimensional singular
set, hence one can also use the Orbifold Theorem to show the results
(see \cite{BMP}).
\end{remark}

\section{Weakly alternating Heegaard diagram}\label{Fexm}
Suppose $M=N_1\cup_{S} N_2$ is a genus two Heegaard splitting. The
disjoint simple closed curves $\alpha_i$, $\beta_i$, $\gamma_i$ in
$S$ bound disks in $N_i$. $\gamma_i$ is a separating curve,
$\alpha_i$ and $\beta_i$ are non-separating and lie in different
sides of $\gamma_i$.

\begin{definition}\label{WAlterD}
We call the diagram
$\{\alpha_1,\beta_1,\gamma_1\}\cup\{\alpha_2,\beta_2,\gamma_2\}$ a
weakly alternating Heegaard diagram if $\gamma_i$ intersects
$\{\alpha_j,\beta_j,\gamma_j\}$ in the cyclic order $$\alpha_j,
\gamma_j, \beta_j, \gamma_j, \alpha_j, \gamma_j, \beta_j, \gamma_j,
\cdots, i\neq j.$$

We call a Heegaard splitting weakly alternating if it admits a
weakly alternating Heegaard diagram.
\end{definition}

\begin{remark}\label{minimal}
Suppose
$\{\alpha_1,\beta_1,\gamma_1\}\cup\{\alpha_2,\beta_2,\gamma_2\}$ is
weakly alternating and $\{\alpha_1,\beta_1,\gamma_1\}$ do not
intersect $\{\alpha_2,\beta_2,\gamma_2\}$ minimally, then there is a
bi-gon in some $\alpha_i\cup\beta_j (i\neq j)$. We can isotopy
$\alpha_i$ or $\beta_j$ to get a new weakly alternating Heegaard
diagram with fewer intersections, and do not affect the
corresponding 3--manifold. Hence following we only consider weakly
alternating Heegaard diagrams with minimal intersections.
\end{remark}

\begin{figure}[h]
\centerline{\scalebox{0.34}{\includegraphics{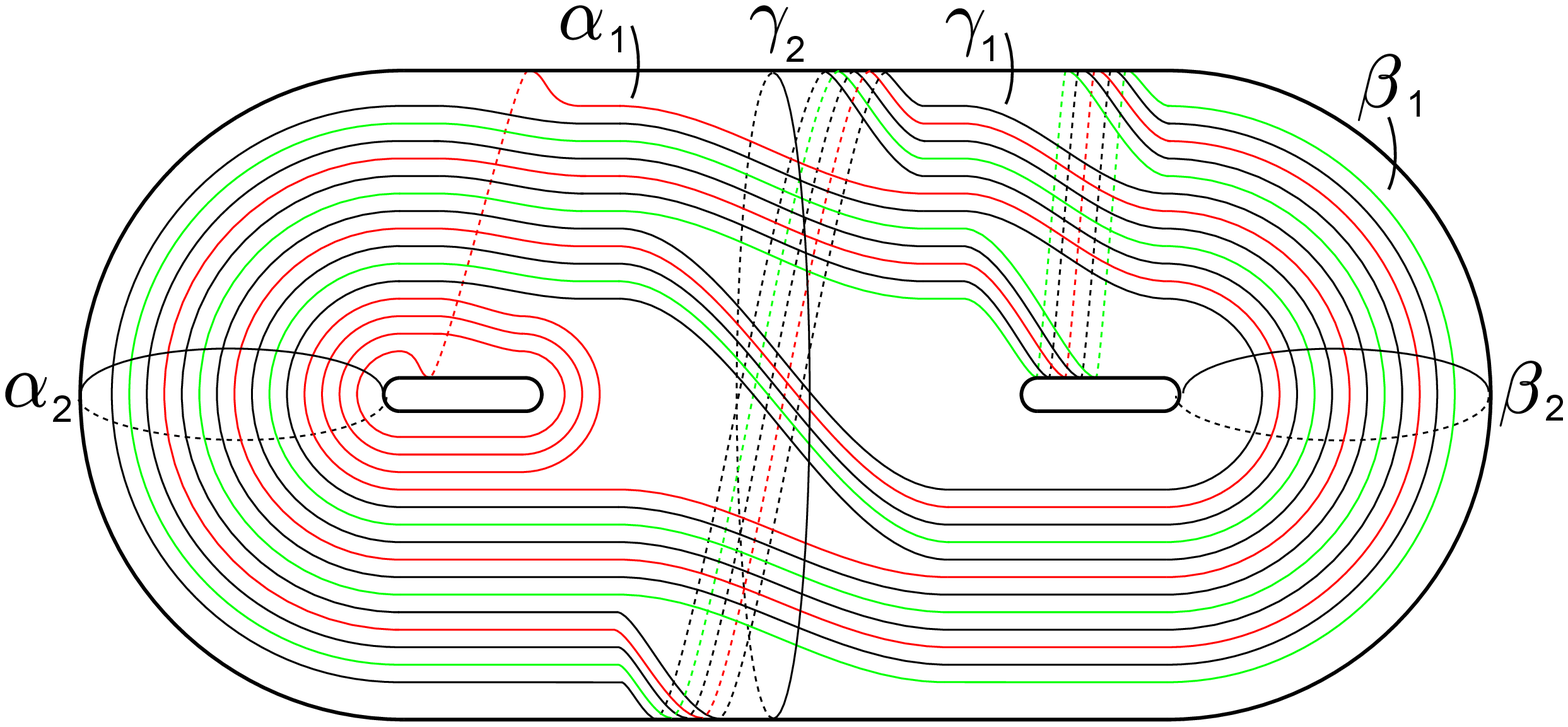}}}
\caption{Weakly alternating Heegaard diagram}\label{fig:w7}
\end{figure}

Clearly an alternating Heegaard diagram is weakly alternating.
Figure \ref{fig:w7} shows a weakly alternating Heegaard diagram
which is not alternating. Latter we will see that this diagram give
us the Poincar\'e's homology 3--sphere $S(-1/2,1/3,1/5)$, which does
not admit any alternating Heegaard splitting. Now we give a proof of
Theorem \ref{wahws}.

\begin{proof}[Proof of Theorem \ref{wahws}]
Similar to the construction part in Section \ref{def}, now we choose
only one point $x_\gamma$ in $\gamma_1\cap\gamma_2$ and add only one
band in $N_i$ connecting $K_i$ and $\gamma_j$, $i\neq j$. Then we
can similarly get $K_i'$, $N(K_i)$, $N(K_i')$ and $f_1$, $f_2$,
$f_3$.

Notice that still we can choose $f_2$ such that the induced maps
$g_1^{-1}$ and $g_2$ satisfy the {\it Expanding condition on $K$},
because in Figure \ref{fig:w1} one can see that the loops have been
drawn longer and after the isotopy the middle arc will also be
longer. Actually we can make a small modification on $K_i'$ as in
Figure \ref{fig:w8}, and correspondingly modify $f_2$ by further
isotopy. Then the expansion on $K$ will be more clear.

\begin{figure}[h]
\centerline{\scalebox{0.45}{\includegraphics{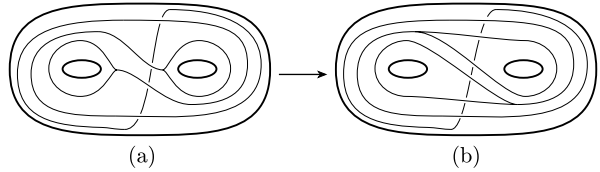}}}
\caption{Expansion of the spine}\label{fig:w8}
\end{figure}

Following the construction {\bf Steps 1, 2, 3} in Section \ref{def},
we can get the required $f$.
\end{proof}

\begin{remark}
By the proof, it is clear that the Williams solenoids derived from
weakly alternating Heegaard splittings are all handcuffs solenoids.
\end{remark}

Suppose
$D(4n;m_1,m_2,m_3)=\{\alpha_1,\beta_1,\gamma_1\}\cup\{\alpha_2,\beta_2,\gamma_2\}$
is an alternating Heegaard diagram. Let $c_i (1\leq i\leq 5)$ be
simple closed curves in $S$ as in Figure \ref{fig:w9}. And let
$t_{c_i}(1\leq i\leq 5)$ denote the Dehn twist along $c_i$ as in
Remark \ref{dehn}.

\begin{figure}[h]
\centerline{\scalebox{0.76}{\includegraphics{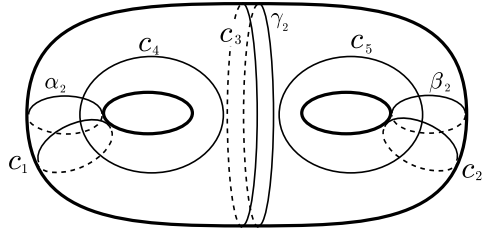}}}
\caption{Simple closed curves in $S$}\label{fig:w9}
\end{figure}

\begin{definition}
Let $l$ and $r$ be two integers. Define $D(4n;m_1[l],m_2[r],m_3)$ to
be the diagram $t_{c_4}^lt_{c_5}^r(\{\alpha_1,\beta_1,\gamma_1\})
\cup \{\alpha_2,\beta_2,\gamma_2\}$, . If $l$ or $r$ is $0$, the
diagram will also be denoted by $D(4n;m_1,m_2[r],m_3)$ or
$D(4n;m_1[l],m_2,m_3)$. $D(4n;m_1[0],m_2[0],m_3)$ is the same as
$D(4n;m_1,m_2,m_3)$, the alternating Heegaard diagram itself.
\end{definition}

\begin{lemma}\label{wsym}
If $D(4n;m_1,m_2,m_3)$ is alternating, then
$D(4n;-m_1,m_2,m_1+m_3)$, $D(4n;m_1',m_2',m_3')$,
$D(4n;m_2,m_1,m_3)$ and $D(4n;-m_1,-m_2,-m_3)$ are alternating, here
$m_i'\equiv m_i (mod\,4n)$. And we have following homeomorphisms:

1. $D(4n;m_1[l],m_2[r],m_3)\simeq D(4n;m_1[l],m_2[r],m_3')$,
$m_3\equiv m_3' (mod\,4n)$.

2. $D(4n;m_1[l],m_2[0],m_3)\simeq D(4n;m_1[l],m_2'[0],m_3)$,
$m_2\equiv m_2' (mod\,4n)$.

3. $D(4n;m_1[l],m_2[r],m_3)\simeq D(4n;m_2[r],m_1[l],m_3)$.

4. $D(4n;m_1[l],m_2[r],m_3)\simeq D(4n;-m_1[-l],-m_2[-r],-m_3)$.

\noindent If further $0<m_1<4n$, we have the following
homeomorphism:

5. $D(4n;m_1[l],m_2[r],m_3)\simeq D(4n;-m_1[l+2],m_2[r],m_1+m_3)$.
\end{lemma}

\begin{proof}
By Theorem \ref{all ad}, one can check directly that the four
diagrams are all alternating Heegaard diagrams. The first four
homeomorphisms can be proved similarly to the proof of Lemma
\ref{sym}. For the last homeomorphism we only need to prove the
following:
$$D(4n;m_1[-2],m_2,m_3)\simeq D(4n;-m_1,m_2,m_1+m_3).$$
This can be shown as in Figure \ref{fig:w10}. Here we only give the
left part of the surface. The notation $x$(or $y$) means that there
are $x$(or $y$) parallel arcs and here $x=m_1$.

\begin{figure}[h]
\centerline{\scalebox{0.7}{\includegraphics{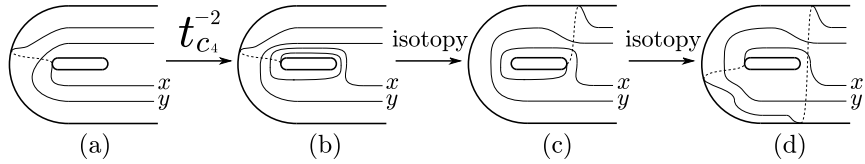}}} \caption{Dehn
twist and isotopy}\label{fig:w10}
\end{figure}

Figure \ref{fig:w10}(a) shows the left part of the diagram
$D(4n;m_1,m_2,m_3)$. After applying the Dehn twist $t_{c_4}^{-2}$ we
get Figure \ref{fig:w10}(b), the left part of
$D(4n;m_1[-2],m_2,m_3)$. This is isotopic to
$D(4n;-m_1,m_2,m_1+m_3)$, via two isotopies as in Figure
\ref{fig:w10}(c) and Figure \ref{fig:w10}(d).
\end{proof}

\begin{lemma}\label{wmirror}
The Dehn twist(half twist), $\pi$--rotation and reflection as in
Figure \ref{fig:7} map a weakly alternating Heegaard diagram to a
weakly alternating Heegaard diagram.
\end{lemma}

\begin{theorem}\label{all wad}
A diagram is a weakly alternating Heegaard diagram if and only if it
has the form $t_{c_1}^{m_4}t_{c_2}^{m_5}(D(4n;m_1[l],m_2[r],m_3))$,
here $n>0$, $m_i(1\leq i\leq 5)$, $l$, $r$ are all integers and
satisfy $(m_1^2-1)l=(m_2^2-1)r=0$.
\end{theorem}

\begin{proof}
{\bf The Only If Part:}

Suppose
$\{\alpha_1,\beta_1,\gamma_1\}\cup\{\alpha_2,\beta_2,\gamma_2\}$ is
a weakly alternating Heegaard diagram on a splitting surface $S$. We
can assume $\{\alpha_2,\beta_2,\gamma_2\}$ to be standard as before
and the curves $\{\alpha_1,\beta_1,\gamma_1\}$ have colors Red,
Green and Black.

Cutting $S$ along $\{\alpha_2,\beta_2,\gamma_2\}$, we get $S_l$ and
$S_r$. Since $\gamma_1$ intersects $\{\alpha_2,\beta_2,\gamma_2\}$
in the cyclic order $\alpha_2$, $\gamma_2$, $\beta_2$,
$\gamma_2,\cdots$, the Black curve must be cut into arcs lying in
$S_l$ and $S_r$ as in Figure \ref{fig:w11}.

\begin{figure}[h]
\centerline{\scalebox{0.63}{\includegraphics{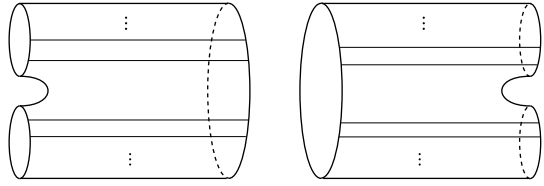}}}
\caption{Black curves in $S_l$ and $S_r$}\label{fig:w11}
\end{figure}

Since $\gamma_2$ intersects $\{\alpha_1,\beta_1,\gamma_1\}$ in the
cyclic order $\alpha_1$, $\gamma_1$, $\beta_1$, $\gamma_1,\cdots$,
the number of intersection points with color Red(Green) must be
even. Cutting $S_l$ and $S_r$ along the Black arcs, since
intersections of $\{\alpha_1,\beta_1,\gamma_1\}$ and
$\{\alpha_2,\beta_2,\gamma_2\}$ are minimal (see Remark
\ref{minimal}), each rectangle piece can contain only one Red(Green)
arc.

\begin{figure}[h]
\centerline{\scalebox{0.63}{\includegraphics{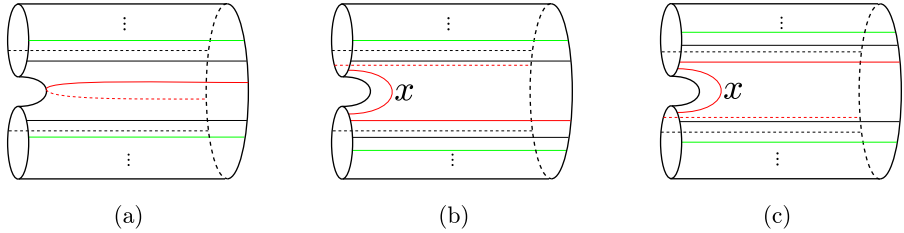}}}
\caption{Three possibilities of colored arcs in
$S_l$}\label{fig:w12}
\end{figure}

If the number of colored arcs in the saddle piece is not $2$, then
modular the Dehn twist along $\alpha_2(\beta_2)$ the pasting way at
$\alpha_2(\beta_2)$ is unique. And in any case all arcs in the
saddle piece will have the same color. We recolor the Red(Green)
curves if it is needed, then $S_l$ should be as in Figure
\ref{fig:w12}. Here the notation $x$ means there are $x$ parallel
arcs. The situation of $S_r$ will be similar.

The original diagram can be obtained from $S_l$ and $S_r$ by pasting
the cuts. Hence we can choose suitable $m_4$, $m_5$, $l$ and $r$,
such that
$t_{c_5}^{-r}t_{c_4}^{-l}t_{c_2}^{-m_5}t_{c_1}^{-m_4}(\{\alpha_1,\beta_1,\gamma_1\})
\cup\{\alpha_2,\beta_2,\gamma_2\}$ is an alternating Heegaard
diagram $D(4n;m_1,m_2,m_3)$. When the number of colored arcs in the
saddle piece of $S_l$ or $S_r$ is not $2$, $l$ or $r$ is not $0$ and
correspondingly $m_1$ or $m_2$ must be $\pm1$. Since
$t_{c_1}^{m_4}t_{c_2}^{m_5}$ preserves the curves
$\{\alpha_2,\beta_2,\gamma_2\}$, we know that
$\{\alpha_1,\beta_1,\gamma_1\}\cup\{\alpha_2,\beta_2,\gamma_2\}$ has
the form as in the Theorem.

{\bf The If Part:}

By the definition of $D(4n;m_1[l],m_2[r],m_3)$,
$D(4n;m_1[0],m_2[0],m_3)$ is always an alternating Heegaard diagram.
Let it be
$\{\alpha_1,\beta_1,\gamma_1\}\cup\{\alpha_2,\beta_2,\gamma_2\}$. If
$m_1$ or $m_2$ is $\pm 1$, then we can apply Dehn twist $t_{c_4}^l$
or $t_{c_5}^r$ on $\{\alpha_1,\beta_1,\gamma_1\}$ to get weakly
alternating Heegraard diagrams. Hence by Lemma \ref{wmirror}, the
diagrams given in the Theorem are all weakly alternating Heegaard
diagrams.
\end{proof}

\section{Manifolds with weakly alternating Heegaard
splittings}\label{WDM}
\begin{definition}
Define $M_i(n;k_1[l],k_2[r],k_3)$ to be the 3--manifold which has a
Heegaard diagram $D(4n;m_1[l],m_2[r],m_3)$ with
$(m_1,m_2,m_3)=\eta_i+4(k_1,k_2,k_3)$, $i=1,2$. Here $\eta_i$ is as
in Definition \ref{admani}. $n>0$, $k_i$, $l$ and $r$ are integers
and $(n,k_1+k_2+2k_3)=1$.
\end{definition}

\begin{lemma}\label{wpare}
If a 3--manifold $M$ admits a weakly alternating Heegaard splitting
but does not admit an alternating Heegraard splitting, then $M$ must
be homeomorphic to one of the following:

1. $M_1(n;0[l],k_2,k_3)$, here $0\leq k_3<n$, $n\leq k_2+2k_3<2n$.

2. $M_1(n;0[l],1[r],k_3)$, here $0\leq k_3<n$.
\end{lemma}

\begin{proof}
By Theorem \ref{all wad} and modular the Dehn twist
$t_{c_1}^{m_4}t_{c_2}^{m_5}$, we only need to consider following
three classes of diagrams: $D(4n;\pm1[l],m_2,m_3)$,
$D(4n;m_1,\pm1[r],m_3)$ and $D(4n;\pm1[l],\pm1[r],m_3)$.

Firstly we consider the first two classes. By Lemma \ref{wsym}, the
3--manifold which can be given by diagrams in these two classes can
also be given by a diagram like $D(4n;1[l],m_2,m_3)$. Then there are
two subclasses: $D(4n;1[l],-3+4k_2,1+4k_3)$ and
$D(4n;1[l],-5+4k_2,2+4k_3)$. But by Lemma \ref{wsym} we have:
\begin{eqnarray*}
&&D(4n;1[l],-3+4k_2,1+4k_3)\\
&\simeq & D(4n;-1[l+2],-3+4k_2,2+4k_3)\\
&\simeq & D(4n;1[-l-2],-5+4(2-k_2),2+4(-1-k_3)).
\end{eqnarray*}
Hence we only need to consider the diagrams
$D(4n;1[l],-3+4k_2,1+4k_3)$ with $(n,k_2+2k_3)=1$, which give us
$M_1(n;0[l],k_2,k_3)$. By Lemma \ref{wsym} again, we can require
$0\leq k_3<n$ and $n\leq k_2+2k_3<2n$.

Similarly for the third class we only need to consider
$D(4n;1[l],1[r],1+4k_3)$ with $(n,1+2k_3)=1$. These diagrams give us
the manifolds $M_1(n;0[l],1[r],k_3)$, and we can require $0\leq
k_3<n$.
\end{proof}

\begin{proposition}\label{wbranch}
The $M_1(n;0[l],k_2,k_3)(M_1(n;0[l],1[r],k_3))$ as in Lemma
\ref{wpare} is a 2--fold branched cover of $S^3$. The branched set
is a three bridge link as in Figure \ref{fig:w13}(a)(Figure
\ref{fig:w13}(b)).

\begin{figure}[h]
\centerline{\scalebox{0.58}{\includegraphics{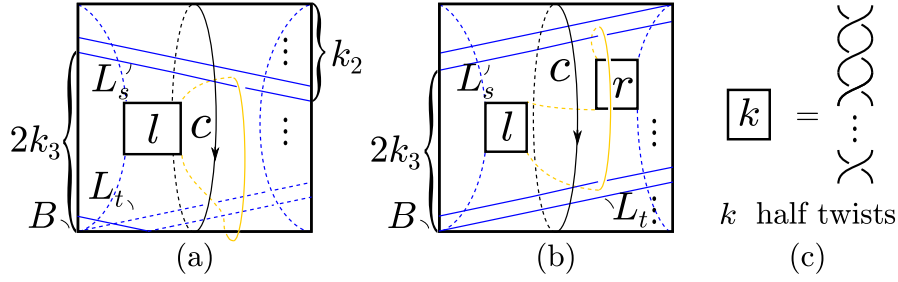}}}
\caption{Branched set}\label{fig:w13}
\end{figure}

The front(back) Blue arcs lying in the surface of the $n\times n$
square pillow have slope $-m/n(m/n)$, $m$ is $k_2+2k_3-n$ in Figure
\ref{fig:w13}(a) and is $1+2k_3-n$ in Figure \ref{fig:w13}(b).
Walking from the point $B$ to right we get the arc $L_t$. Walking
against the oriented circle $c$ from $L_t$ by $2k_3$ we get the arc
$L_s$. And then the position of the Yellow arc can be determined.
The $k$--box denotes two parallel arcs with $k$ half twists.
Over-crossings are from lower left to upper right if $k>0$, and
upper left to lower right if $k<0$.
\end{proposition}

\begin{proof}
We only prove the case of $M_1(n;0[l],k_2,k_3)$. The case of
$M_1(n;0[l],1[r],k_3)$ can be proved similarly.

Firstly consider the 2--fold branched cover from
$M_1(n;0[0],k_2,k_3)$ to $S^3$. Figure \ref{fig:w14} shows us the
position of the Dehn twist curve $c_4$ in $N_2$ and its image
$c_4/\tau$ in $N_2/\tau$. Here $\tau$ is the involution.

\begin{figure}[h]
\centerline{\scalebox{0.65}{\includegraphics{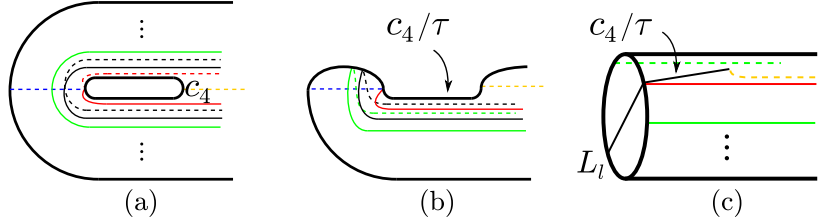}}}
\caption{Position of $c_4/\tau$}\label{fig:w14}
\end{figure}

Figure \ref{fig:w14}(a) gives the left part of $N_2$ and Figure
\ref{fig:w14}(b) gives its quotient. This quotient can also be given
by modular the reflection along the line $L_l$ in Figure
\ref{fig:w14}(c). One can compare Figure \ref{fig:w14}(c) to Figure
\ref{fig:15}(a).

By the proof of Proposition \ref{branch}, one can see that the
quotient $M_1(n;0[0],k_2,k_3)/\tau$ has branched set as in Figure
\ref{fig:w15}(a).

\begin{figure}[h]
\centerline{\scalebox{0.58}{\includegraphics{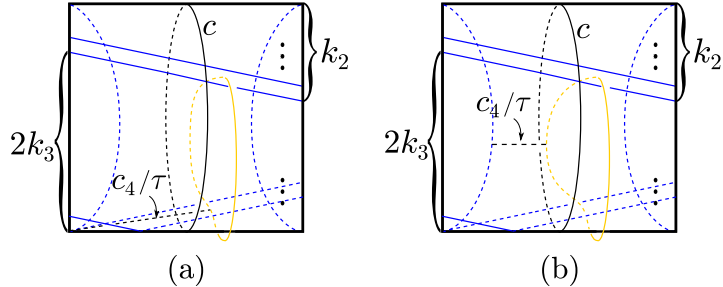}}}
\caption{Isotopy $c_4/\tau$ into the square pillow}\label{fig:w15}
\end{figure}

Notice that a Dehn twist along $c_4$ in $S$ will induce a half twist
around $c_4/\tau$ in $S/\tau$. Hence for $M_1(n;0[l],k_2,k_3)/\tau$,
when we paste $N_1/\tau$ to $N_2/\tau$, the gluing map will be
different from the case of $M_1(n;0[0],k_2,k_3)/\tau$ by $l$ half
twists around $c_4/\tau$.

We can require these $l$ half twists happened in a small
neighborhood of $c_4/\tau$, and isotopy $c_4/\tau$ and its
neighborhood into the square pillow as in Figure \ref{fig:w15}(b).
Then we will get the picture as in Figure \ref{fig:w13}(a).
\end{proof}

\begin{proposition}
We have the following homeomorphisms:

1. $M_1(n;0[l],n+m,0)\simeq L(n,m)\,\#\,L(l,1)$, $0\leq m<n$.

2. $M_1(n;0[l],n-m,m)\simeq S^2(-1/2,1/(l+2),m/n)$, $0<m<n$.

3. $M_1(n;0[l],1[r],0)\simeq S^2(1/l,1/s,1/n)$, $n>0$.
\end{proposition}

\begin{proof}
This proof is similar to the proof of Proposition \ref{Idmfd} and
the same argument as in Figure \ref{fig:17} will be used.

For $M_1(n;0[l],n+m,0)$ the branched set is isotopic to the link as
in Figure \ref{fig:w16}(a). It is clear that the link is a connected
sum. The 2--fold branched cover of this link is the connected sum of
two Lens spaces.

\begin{figure}[h]
\centerline{\scalebox{0.5}{\includegraphics{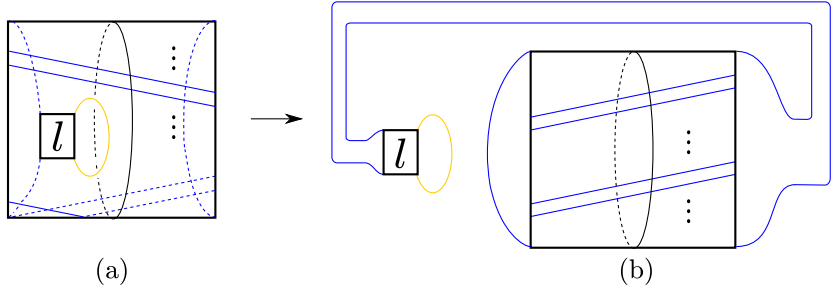}}}
\caption{Connected sum of two bridge links}\label{fig:w16}
\end{figure}

For $M_1(n;0[l],n-m,m)$ the branched set is isotopic to the link as
in Figure \ref{fig:w17}(a). Pushing the Yellow arc into the square
pillow, we will get Figure \ref{fig:w17}(b).

\begin{figure}[h]
\centerline{\scalebox{0.5}{\includegraphics{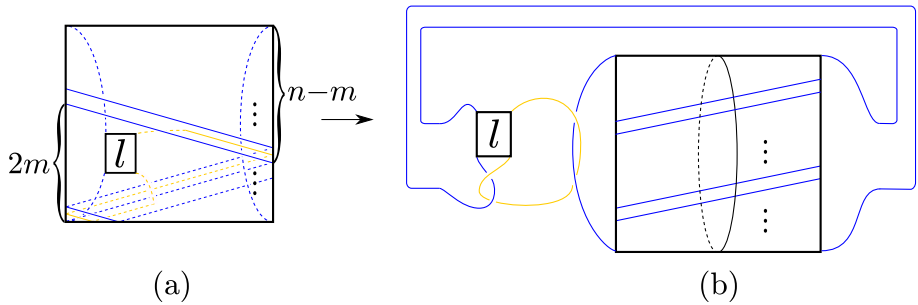}}}
\caption{Montesinos link or connected sum}\label{fig:w17}
\end{figure}

When $l\neq -2$, it is a Montesinos link with three rational tangles
having parameters $(1,l)$, $(-1,2)$ and $(m,n)$. When $l=-2$, this
link is a connected sum of a two bridge link and a Hopf link. The
corresponding 3--manifold is a Seifert fibred space or a connected
sum. The manifold can be presented uniformly as
$S^2(-1/2,1/(l+2),m/n)$.

For $M_1(n;0[l],1[r],0)$ the branched set is isotopic to the link as
in Figure \ref{fig:w18}(a). When we take a half twist at the right
side of the square pillow, we will get Figure \ref{fig:w18}(b). And
we can further isotopy it to Figure \ref{fig:w18}(c).

\begin{figure}[h]
\centerline{\scalebox{0.58}{\includegraphics{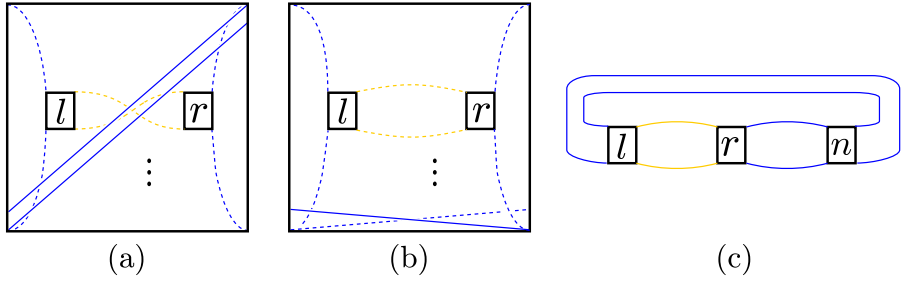}}}
\caption{Pretzel link or connected sum}\label{fig:w18}
\end{figure}

Clearly when $l\neq 0$ and $r\neq 0$, it is a Montesinos link with
three rational tangles having parameters $(1,l)$, $(1,r)$ and
$(1,n)$. This is also a Pretzel link. When $l\neq 0$ or $r\neq 0$,
it is a connected sum. The corresponding 3--manifold can be a
Seifert fibred space, a connected sum of Lens spaces or a connected
sum of a Lens space and $S^1\times S^2$. The manifold can be
presented uniformly as $S^2(1/l,1/r,1/n)$.
\end{proof}

\begin{proposition}\label{figure8}
$M_2(5;0[l],4,1)\simeq S^3_{l-2/1}(4_1)(\simeq M_1(5;0[-l-2],3,3))$.
\end{proposition}

\begin{proof}
By the proof of Proposition \ref{branch} and \ref{wbranch}, it is
not hard to see the corresponding link of $M_2(n;0[l],k_2,k_3)$ is
as in Figure \ref{fig:h4}(a). Hence the corresponding link of
$M_2(5;0[l],4,1)$ is as in Figure \ref{fig:h4}(b), and it is
isotopic to the link as in Figure \ref{fig:h4}(c).

\begin{figure}[h]
\centerline{\scalebox{0.5}{\includegraphics{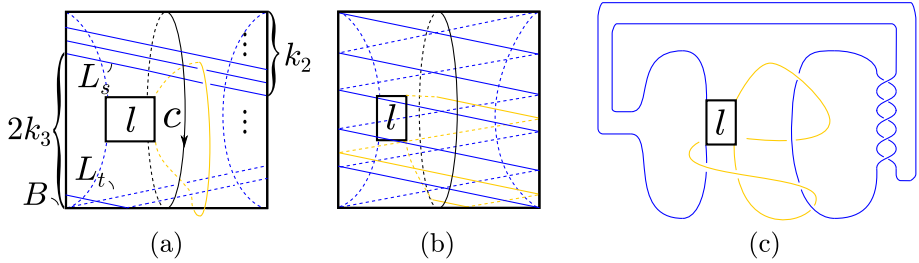}}} \caption{The
quotient $M_2(5;0[l],4,1)/\tau$}\label{fig:h4}
\end{figure}

On the boundary of the $l$--box we draw a green arc connecting two
singular points and winding around the box $l/2$ rounds. It can be
obtained from the trivial case by $l$ half twists, see Figure
\ref{fig:h5}(a) for the case $l=5$. Clearly the 2--fold branched
cover of the box is a solid torus. And the 2--fold branched cover of
this green arc is a green circle, which bounds a disk in the solid
torus.

If we replace the $l$--box by a box containing two parallel
horizontal singular arcs, then the singular set will be as in Figure
\ref{fig:h5}(b), which is a trivial knot. The new box can be thought
as a regular neighborhood of a regular arc, and the green arc winds
around the regular arc $l/2$ rounds. We can isotopy this picture to
the position as in Figure \ref{fig:h5}(c).

\begin{figure}[h]
\centerline{\scalebox{0.5}{\includegraphics{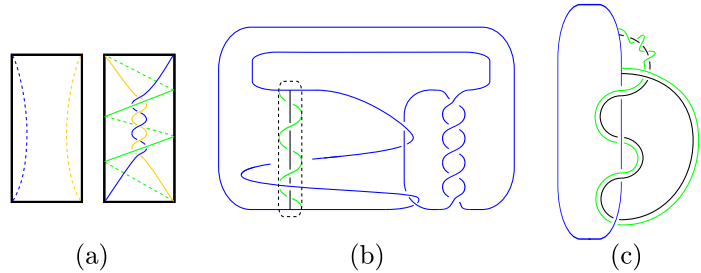}}}
\caption{Surgery on $M_2(5;0[l],4,1)/\tau$}\label{fig:h5}
\end{figure}

Then it is easy to see its 2--fold branched cover is $S^3$ and the
2--fold branched cover of the regular arc is a figure eight knot, as
in Figure \ref{fig:h6}(a).

\begin{figure}[h]
\centerline{\scalebox{0.5}{\includegraphics{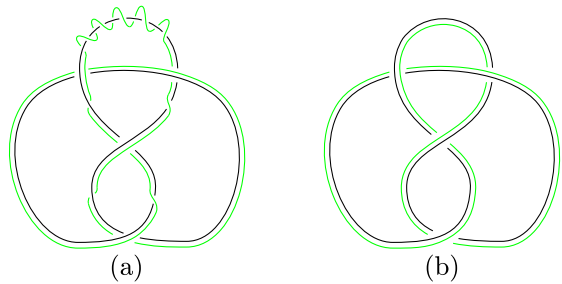}}} \caption{The
figure eight knot}\label{fig:h6}
\end{figure}

Now the green circle winds around the regular circle $l-2$ rounds.
Removing the regular neighborhood of the figure eight knot, since
the circle which bounds a Seifert surface in the complement is
parallel to the knot as in Figure \ref{fig:h6}(b), we know that
$M_2(5;0[l],4,1)$ is the $l-2$ surgery on the figure eight knot.
\end{proof}

\begin{remark}
1. The figure eight knot has exactly $10$ exceptional slops, namely
$\infty$ and $-4\leq p/1\leq 4$. Other $S^3_{p/1}(4_1)$ are all
hyperbolic. The exceptional cases are listed below (see \cite{Th}).
\begin{itemize}
\item $S^3_{\infty}(4_1)\simeq S^3$.
\item $S^3_{0/1}(4_1)$ is the $T^2$--bundle over $S^1$ with monodromy
$\bigl(\begin{smallmatrix} 1 & 1\\1 & 2\end{smallmatrix}\bigr)$. It
admits the {\bf Sol} geometry.
\item $S^3_{\pm1/1}(4_1)\simeq S^2(-1/2,1/3,1/7)$.
\item $S^3_{\pm2/1}(4_1)\simeq S^2(-1/2,1/4,1/5)$.
\item $S^3_{\pm3/1}(4_1)\simeq S^2(-2/3,1/3,1/4)$.
\item $S^3_{\pm4/1}(4_1)$ is the union of the trefoil knot complement and
the twisted $I$--bundle over the Klein bottle. It contains an
incompressible torus.
\end{itemize}

2. One can also show that $M_2(n;0[1],n-3,2)\simeq
S^3_{n-2/1}(4_1)(n\geq 5)$ by a similar way. Compare it to
Proposition \ref{hyperbolicM}.
\end{remark}

\begin{remark}

1. $S^1\times S^2 \,\#\, S^1\times S^2$ is a genus two 3--manifold
which has no weakly alternating Heegaard splitting. Otherwise its
$\pi_1\cong \mathbb{Z}\ast\mathbb{Z}/H$ with $H$ nontrivial. But
$\mathbb{Z}\ast\mathbb{Z}/H \ncong \mathbb{Z}\ast\mathbb{Z}$ because
$\mathbb{Z}\ast\mathbb{Z}$ is Hopfian. Actually $S^1\times
S^2\,\#\,S^1\times S^2$ does not admit any automorphism $f$ with
$\Omega(f)$ consists of Williams solenoids, whose defining
neighborhoods having genus $g\leq 2$. This is similar to the fact
that $S^1\times S^2$ does not admit any automorphism $f$ with
$\Omega(f)$ consists of Smale solenoids.

2. By Section \ref{consf} and Theorem \ref{wahws}, we see that
globally there can be many non-homeomorphic Williams
solenoids(handcuffs solenoids) in a given 3--manifold, as the
non-wondering sets of non-conjugate automorphisms. The following
question is natural, which have been studied in \cite{MY2} in the
case of Smale solenoids.

{\bf Question :} Given a 3--manifold $M$, what kind of Williams
solenoids (with defining neighborhoods having genus $g\leq 2$) can
be globally realized as attractors in $M$? And how many of them?

3. We have shown that half of Prism manifolds admit automorphisms
$f$ with $\Omega(f)$ consist of two Williams solenoids. Hence it is
natural to ask what about the other half, namely $P(m,n)$ with
$0<n<m$? In the case of $S^3_{l/1}(4_1)$ one can ask what about
other surgeries? Generally we can ask the following question.

{\bf Question :} Does a 3--manifold in the following classes (all
having Heegaard genus two) admit an automorphism whose non-wondering
set consists of Williams solenoids (with defining neighborhood
having genus $g\leq 2$)?
\begin{itemize}
\item Seifert fibred spaces $S^2(a,b,c)$.
\item Surgeries on two bridge knots.
\end{itemize}

4. The manifolds as in Lemma \ref{pare} and \ref{wpare} may give
homeomorphic ones. But on the other hand, they can give many kinds
of 3--manifolds. We wonder how to classify them and get more
familiar genus two 3--manifolds admitting dynamics $f$ such that
$\Omega(f)$ consist of solenoid attractors and repellers.
\end{remark}

{\bf Acknowledgement.} The authors would like to thank Professor
Shicheng Wang and Xiaoming Du for their many helpful discussions.

\bibliographystyle{amsalpha}

\end{document}